\documentclass[english]{article} 
\usepackage[utf8]{inputenc}
\usepackage[T1]{fontenc}
\usepackage{lmodern}
\usepackage[a4paper]{geometry}
\usepackage{enumitem}
\usepackage{verbatim}
\usepackage{graphicx}
\usepackage{amsmath}
\usepackage{amsthm}
\usepackage{amsfonts}
\usepackage{amssymb}
\usepackage{amsmath}
\usepackage{mathrsfs}
\usepackage[english]{babel}
\usepackage[citecolor=blue,colorlinks=true]{hyperref}

\theoremstyle{plain}
\newtheorem{theorem}{Theorem}[section]
\newtheorem{proposition}[theorem]{Proposition}
\newtheorem{lemma}[theorem]{Lemma}
\newtheorem{corollary}[theorem]{Corollary}

\theoremstyle{definition}
\newtheorem{definition}{Definition}[section]

\theoremstyle{remark}
\newtheorem{remark}{Remark}[section]

\numberwithin{equation}{section}
\newcommand{\N}{\mathbb{N}}

\newcommand{\R}{\mathbb{R}}

\def\ocirc#1{\ifmmode\setbox0=\hbox{$#1$}\dimen0=\ht0
    \advance\dimen0 by1pt\rlap{\hbox to\wd0{\hss\raise\dimen0
    \hbox{\hskip.2em$\scriptscriptstyle\circ$}\hss}}#1\else
    {\accent"17 #1}\fi} 

\newcommand{\eps}{\varepsilon}
\newcommand{\F}{\mathcal{F}}
\newcommand{\G}{\mathcal{G}}
\newcommand{\PP}{\mathbb{P}}
\newcommand{\dual}[2]{\langle #1, #2\rangle}
\newcommand{\LL}{\mathscr{L}}
\newcommand{\E}{\mathbb{E}}
\newcommand{\T}{\mathbb{T}}

\DeclareMathOperator{\divv}{div}
\DeclareMathOperator{\Lip}{Lip}
\def\topi{\stackrel{\pi}{\longrightarrow}}

\begin{document}

\title{Diffusion-approximation for a kinetic equation with perturbed velocity redistribution process}
\author{Nils Caillerie\thanks{1, place De Lattre de Tassigny, V\'etraz-Monthoux  BP241, 74106 Annemasse Cedex, France} and Julien Vovelle\thanks{Univ Lyon, CNRS, ENS Lyon, UMR5669, Unit\'e de Math\'ematiques Pures et Appliqu\'ees, F-69364 Lyon, France}}
\maketitle

\begin{abstract} We derive the hydrodynamic limit of a kinetic equation with a stochastic, short range perturbation of the velocity operator. Under some mixing hypotheses on the stochastic perturbation, we establish a diffusion-approximation result: the limit we obtain is a parabolic stochastic partial differential equation on the macroscopic parameter, the density here.
\end{abstract}
\medskip 


\normalsize

\tableofcontents

\section{Introduction}\label{introduction}
Let $\T^d$ denote the $d$-dimensional torus. Let $V$ be a bounded domain of $\R^d$, say $V\subset \bar{B}_{\R^d}(0,1)$, and let $\nu$ be a probability measure on $V$. We consider the following kinetic random equation:
\begin{equation}\label{eq:1}
\partial_{t}f^{\varepsilon}+\frac{v}{\varepsilon}\cdot\nabla_{x}f^{\varepsilon}=\frac{1}{\varepsilon^{2}}\left(M\rho^{\varepsilon}-f^{\varepsilon}\right)+\frac{1}{\varepsilon^{2}}\rho^{\varepsilon}v\cdot\nabla_{x} \bar{\mathtt{m}}^{\eps}_t,\quad\left(t,x,v\right)\in\mathbb{R}_{+}\times\mathbb{T}^{d}\times V,
\end{equation}
with initial condition
\begin{equation}\label{CI}
f^{\eps}(0)=f^{\eps}_\mathrm{in}\in L^2(\T^d \times V).
\end{equation}
In \eqref{eq:1}, $\rho^\eps$ is the density associated to $f^\eps$:
\begin{equation}\label{def:rhof}
\rho^\eps=\int_V f^\eps(v)d\nu(v).
\end{equation}	
The parameter $\eps>0$ is small and we will study the limit of \eqref{eq:1} when $\eps\to 0$. The random character of \eqref{eq:1} comes from the factor $v\cdot\nabla_{x}\bar{\mathtt{m}}^{\eps}_t$. In this term 
\begin{equation}\label{mrescaled}
\bar{\mathtt{m}}^{\eps}_t(x)=\bar{\mathtt{m}}_{\eps^{-2}t}(x),
\end{equation}
where $(\bar{\mathtt{m}}_t)$ is a stationary stochastic process over $C^r(\T^d)$, $r$ large enough (\textit{cf.} Section~\ref{sec:driving}). The function $M$ is a probability density function on $(V,\nu)$. We will assume that $M$ is bounded from above and from below: 
\begin{equation}\label{HypM}
\alpha\leq M(v)\leq\alpha^{-1},\text{ for }\nu\text{ a.e. }v\in V,
\end{equation} 
where $\alpha\in(0,1)$. Due to \eqref{mrescaled}, Equation (\ref{eq:1}) is obtained from the change of scale $f^{\eps}(t,x,v)=f(\eps^2 t,x,v)$, where $f$ is a solution to
\begin{equation}
\partial_{t}f+\eps v\cdot\nabla_{x}f=\left(M\rho-f\right)+\rho v\cdot\nabla_{x}\bar{\mathtt{m}}_t,\quad\left(t,x,v\right)\in\mathbb{R}_{+}\times\mathbb{T}^{d}\times V.\label{eq:2}
\end{equation}
Let $\mathrm{Leb}_d$ denote the Lebesgue measure on $\T^d$ and let $\mathrm{Leb}_d\times\nu$ denote the product measure on $\T^d\times V$. When the initial data for \eqref{eq:2} is a probability density on $\T^d\times V$ with respect to $\mathrm{Leb}_d\times\nu$, then so is $f(t)$ for all $t>0$. The associated probability is the law of a certain space-velocity process $(X(t),V(t))$. The evolution of $(X(t),V(t))$ is given as a coupling between the equation $\dot{X}(t)=\eps V(t)$ and the evolution of $V$ according to a jump process having the following parameters: the waiting times are  exponential laws of parameter one; at each waiting time $T$, the velocity is redistributed according to a probability measure of density
\begin{equation}\label{checkM}
\check{M}\colon v\mapsto M(v)+v\cdot\nabla_x\bar{\mathtt{m}}_T(X(T))
\end{equation}
with respect to $\nu$. In Remark~\ref{rk:runandtumble} below, we account for a possible application of this framework to the modelling of the motion by run-and-tumble of some given agents.
When $\eps=0$, \eqref{eq:2} reduces to the equation
\begin{equation}
\partial_{t}f=\left(M\rho-f\right)+\rho v\cdot\nabla_{x}\bar{\mathtt{m}}_t,\quad\left(t,x,v\right)\in\mathbb{R}_{+}\times\mathbb{T}^{d}\times V.\label{eq:3}
\end{equation}
Under some mixing hypotheses on the process $(\bar{\mathtt{m}}_t)$, Equation~\eqref{eq:3} has a unique invariant measure. This invariant measure is the law of a particular solution $\rho(x)\bar{M}_t$ (note that $x$ is a parameter in \eqref{eq:3}),
where
\begin{equation}\label{def:barMt}
\bar{M}_t=M+v\cdot\nabla_x\bar{\mathtt{w}}_t,\quad \bar{\mathtt{w}}_t := \int_{-\infty}^t e^{-(t-s)}\bar{\mathtt{m}}_s ds.
\end{equation} 
We refer to Section~\ref{sec:main-generator}, where the justification of the explicit formula~\eqref{def:barMt} is given. Consider the evolution corresponding to \eqref{eq:2}, when the initial datum is close\footnote{actually, it is not necessary to start close to equilibrium, since the dynamics of \eqref{eq:2} brings the solution close to local equilibrium in short time, see the bound on the entropy dissipation in 
\eqref{relativeentropy}-\eqref{entropydiss}
} 
to the equilibrium $\rho(x)\bar{M}_0(v)$. On the long time scale $\eps^{-2}t$, we show that the rescaled unknown $f^\eps$ solution to \eqref{eq:1} is close to a local equilibrium $\rho_t\bar{M}_{\eps^{-2}t}$, and we give the evolution for the macroscopic parameter $(\rho_t)$. The fact that $t\mapsto\eps^{-2}t$ is the pertinent change of time scale is due to the structure of \eqref{eq:2} and to the following cancellation and normalization properties of $M$ and $(V,\nu)$:
\begin{equation}\label{vdv}
\int_V Q(v) d\nu(v)=0,\quad \int_V M(v) d\nu(v)=1,
\end{equation}
where $Q(v)\in\{v_i,v_iM(v),v_i v_j v_k M(v)\}$, for $i,j,k\in\{1,\ldots,d\}$. We will also assume the following non degeneracy hypothesis:
\begin{equation}\label{nuND}
\forall\xi\in\R^d,\;(v\cdot\xi=0\mbox{ for }\nu-\mbox{a.e. }v\in V)\Rightarrow\xi=0.
\end{equation}
Our precise statement is given in Section~\ref{sec:mainresult}, after the following Section~\ref{sec:driving} where the nature of the random process $(\bar{\mathtt{m}}_t)$ is specified. General comments about the theme of this paper and related works are also postponed to Section~\ref{sec:mainresult}.

\subsection{The driving random term}\label{sec:driving}

The terminology about Markov processes used here and below is explained in Appendix~\ref{sec:appendix}.
Let $(\Omega,\mathcal{F},\PP)$ be a probability space.  Let $r\in\N$, $r>2+d/2$, and let $F=C^r(\T^d)$ be the Banach space with norm
$$
\|m\|_F=\sup\left\{|D^k m(x)|;x\in\T^d,0\leq |k|\leq r\right\},
$$
where $|k|=k_1+\dotsb+k_d$. Let $\mathtt{F}$ be a convex Borel subset of $F$. Let $\G$ denote the $\sigma$-algebra induced by the Borel $\sigma$-algebra of $F$ on $\mathtt{F}$. We consider a stationary, homogeneous, c{\`a}dl{\`a}g Markov process $(\bar{\mathtt{m}}_t)_{t\geq 0}$ with state space $\mathtt{F}$. We denote by $\mathtt{A}$ the infinitesimal generator of $(\bar{\mathtt{m}}_t)_{t\geq 0}$ and by $P(t,m,B)$ the transition kernel.
Up to a modification of the probability space, and by identification of versions of the processes, we are given, for each random variable $\mathtt{m}_0$ on $\mathtt{F}$, some processes $\mathtt{m}(t,s;\mathtt{m}_0)$ with transition function $P$, satisfying $\PP(\mathtt{m}(s,s;\mathtt{m}_0)\in B)=\PP(\mathtt{m}_0\in B)$. We can also assume that $(\bar{\mathtt{m}}_t)$ is defined for all $t\in\R$ (see the beginning of \cite[Section~2]{DebusscheVovelle20} for the justification of these assertions). We will then use the notations $\mathtt{m}(t,s;m)$ when $\mathtt{m}_0=m$ almost surely, and set $\mathtt{m}(t;m)=\mathtt{m}(t,0;m)$. We will also denote by $\lambda$ the law of $(\bar{\mathtt{m}}_t)_{t\geq 0}$, which is an invariant measure.
Our first hypothesis is that the process is almost surely bounded: there exists $\mathtt{b}\geq 0$ such that 
\begin{equation}\label{BallR}
\mathtt{F}\subset\bar{B}(0,\mathtt{b}),
\end{equation}
where $\bar{B}(0,\mathtt{b})$ is the closed ball of radius $\mathtt{b}$ centred at the origin in $F$.
We will assume that $\mathtt{b}$ is sufficiently small in order to ensure that the matrix $K^*$ defined by \eqref{Kflat} is positive:
\begin{equation}\label{Rsmall}
\mathtt{b}\leq\frac{\alpha}{4}.
\end{equation}

Our second hypothesis is that the invariant measure $\lambda$ is centred:
\begin{equation}\label{mcentred}
\int_\mathtt{F} m\ d\lambda(m)=\E\left[\bar{\mathtt{m}}_t\right]=0.
\end{equation}

Our third hypothesis is a mixing hypothesis: we assume that there exists a continuous, non-increasing, positive and integrable function $\gamma_\mathrm{mix}\in L^1(\R_+)$ such that, for all probability measures $\mu$, $\mu'$ on $\mathtt{F}$, for all random variables $\mathtt{m}_0$, $\mathtt{m}_0^\prime$ on $\mathtt{F}$ of law $\mu$ and $\mu'$ respectively, there is a coupling $((\hat{\mathtt{m}}_t)_{t\geq 0},(\hat{\mathtt{m}}^\prime_t)_{t\geq 0})$ of $(\mathtt{m}(t;\mathtt{m}_0))_{t\geq 0},(\mathtt{m}(t;\mathtt{m}_0^\prime))_{t\geq 0})$ such that 
\begin{equation}\label{mixCoupled}
\E\|\hat{\mathtt{m}}_t-\hat{\mathtt{m}}_t^\prime\|_F\leq\mathtt{b}\gamma_\mathrm{mix}(t),
\end{equation}
for all $t\geq 0$. 
Let $\theta\colon F\to\R$ be a continuous function, bounded on bounded sets of $F$. A consequence of \eqref{BallR} is that, for $\alpha> 0$, the \textit{resolvent}
\begin{equation}\label{resolventA}
U_\alpha\theta(m):=\int_0^\infty e^{-\alpha t}\E\theta(\mathtt{m}(t;m))dt,\quad m\in\mathtt{F},
\end{equation}
is well defined. In the limiting case $\alpha=0$, let us assume furthermore that $\theta$ is Lipschitz-continuous on bounded sets of $F$ and satisfies the cancellation condition $\dual{\theta}{\lambda}=0$. By Appendix~\ref{sec:app5}, the resolvent $U_0\theta(m)$ is also well-defined. These considerations show that, setting
\begin{equation}\label{defchin}
\chi(m)=K(1)\nabla_x m=\int_V v v_i\partial_{x_i} m d\nu(v)\in\R^d,\quad m\in F,
\end{equation}
the quantity $U_\alpha\chi(m)$ is well defined as an element of $\left[C^{r-1}(\T^d)\right]^d$ and satisfies the estimate
\begin{equation}\label{boundRchi}
\|U_\alpha\chi(m)\|_{C^{r-1}(\T^d)}\leq \mathtt{b}\|\gamma_\mathrm{mix}\|_{L^1(\R_+)},
\end{equation}
for all $\alpha\geq 0$ and $m\in\mathtt{F}$. Our last assumption (which is used to establish Proposition~\ref{prop:SquareOrderone}) is the following one. We consider a bounded li\-ne\-ar functional $\Lambda$, with norm $\|\Lambda\|$, on $\left[C^{r-1}(\T^d)\right]^d$. By composition, we may consider the functional $\Lambda\circ U_0\circ\chi\colon m\mapsto\Lambda\left[(U_0\chi(m))\right]$. We will assume the following bound: there exists a constant $C^0_\mathtt{b}\geq 0$ such that
\begin{equation}\label{AR0}
|\mathtt{A}\left|\Lambda\circ U_0\circ\chi\right|^2(m)|\leq C^0_\mathtt{b}\|\Lambda\|^2,
\end{equation}
for all $m\in\mathtt{F}$. 

\begin{definition}[Admissible pilot process] Let $\mathtt{F}$ be a closed subset of $F$. Let $(\bar{\mathtt{m}}_t)_{t\geq 0}$ be a c{\`a}dl{\`a}g, stationary, homogeneous Markov process with state space $\mathtt{F}$ and infinitesimal generator $\mathtt{A}$. We say that $(\bar{\mathtt{m}}_t)_{t\geq 0}$ is an admissible pilot process if the conditions \eqref{BallR}, \eqref{Rsmall}, \eqref{mcentred}, \eqref{mixCoupled}, \eqref{AR0}  above are satisfied and if $(\bar{\mathtt{m}}_t)_{t\geq 0}$ is stochastically continuous.
\label{def:admbarm}\end{definition}

\subsection{Some examples}\label{sec:example-driving}

Some examples of admissible pilot processes are provided by jump processes (\textit{e.g.} the jump process described in Section~2C of \cite{DebusscheVovelle20}) or diffusion processes. Let us give an example in this second class of processes. Consider an orthonormal basis $(n_j)_{j\geq 1}$ of $L^2(\T^d)$ made of smooth functions. We fix an integer $N>0$ and set
\begin{equation}\label{mtdiffusion}
\bar{\mathtt{m}}_t=\sum_{j=1}^N a_j \bar{Y}^j_t n_j,
\end{equation}
where $\bar{Y}^1,\bar{Y}^2,\dotsc$ are some i.i.d. processes with state space the interval $(-1,1)$ and $a_1,a_2,\dotsc$ some non-trivial real numbers converging fast enough to zero. When $N$ is finite, the definition \eqref{mtdiffusion} gives a process $(\bar{\mathtt{m}}_t)$ with state space $\mathtt{F}_N$ that has finite dimension. Indeed, $\mathtt{F}_N$ consists of all functions $m$ in the vector space generated by $n_1,\dotsc,n_N$ such that $|\dual{m}{a_j^{-1} n_j}_{L^2(\T^d)}|\leq 1$ for all $1\leq j\leq N$. By letting $N\to+\infty$ (we will not give the details of this procedure), one gets an admissible process with infinite-dimensional state space. Replacing $\bar{Y}^j_t$ by an element $y_j\in (-1,1)$ defines a function $H^\sharp\colon (-1,1)^N\to F$ which is a continuous bijection onto its image, with inverse $m\mapsto (\dual{m}{a_j^{-1} n_j}_{L^2(\T^d)})$. In this way, the process $\bar{\mathtt{m}}_t$ is deduced from $(\bar{Y}_t)$ by a change of coordinates (see Appendix~\ref{sec:app4}). The properties of $(\bar{\mathtt{m}}_t)$ essentially depend on the properties of one single component process $\bar{Y}$ and we will focus on the construction of such an adequate process. We want to build, therefore, a diffusion process with some ergodic properties and some bounded state space. Instances of diffusion processes with (one-dimensional) bounded state spaces are furnished by processes reflected or killed at the boundary, or Sturm-Liouville Markov processes with a drift that is singular at the boundary, \cite[Chapter~2]{BakryGentilLedoux14}. We choose to give an example of Sturm-Liouville Markov process $Y$ deduced from a change of coordinate applied to a diffusion process defined on the whole line. Let $(X_t)$ be the one-dimensional Markov process described by the SDE
\begin{equation}\label{SDEalpha}
dX_t=-f(X_t) dt+dB_t,\quad f(x):=\max(1,|x|^\alpha)x.
\end{equation}
In \eqref{SDEalpha}, $\alpha$ is a strictly positive number. For $\alpha=0$, we obtain the Ornstein-Uhlenbeck process. We need to restrict ourselves to the case of strictly positive $\alpha$ however, to ensure some uniformity with respect to the initial data in the construction of a coupling (see \eqref{egcoupling} below). The process $(X_t)$ is stationary when initialized with the law $\lambda^X$ of density given, up to a renormalizing factor, by
\[
d\lambda^X(x)\propto \exp\left(-2F(x)\right)dx,\quad F(x)=\int_0^x f(y) dy.
\]
By proceeding as in the proof of \cite[Theorem 1.1, p.157]{EthierKurtz86}, we may construct a stationary process $(\bar{X}_t)_{t\in\R}$ indexed by $\R$ such that, for every $t_0<t_1<\dotsb<t_k$, the $k$-uplet
\begin{equation}\label{bar0bar}
(\bar{X}_{t_0},\dotsc,\bar{X}_{t_k})\mbox{ and }(\bar{X}^0_{t_0+s_0},\dotsc,\bar{X}^0_{t_k+s_0})
\end{equation}
have the same law, where $s_0$ is any real such that $t_0+s_0\geq 0$. By some standard estimates based on the integral form of \eqref{SDEalpha} one can show that the increments of $(\bar{X}^0_t)_{t\in\R_+}$ satisfy the estimate
\begin{equation}\label{incrementbarX0}
\E|\bar{X}^0_t-\bar{X}^0_s|^4\leq C|t-s|^2,\quad t,s\in [0,T],
\end{equation}
where $C$ is a constant depending on $\alpha$ and $T$. By identity of the laws in \eqref{bar0bar}, we see that we can apply the Kolmogorov criterion to $(\bar{X}_t)_{t\in\R}$ on any compact interval of $\R$. We may therefore construct a version of $(\bar{X}_t)_{t\in\R}$ with continuous trajectories, and it is this version that we consider now. The estimate \eqref{incrementbarX0} also shows that $(\bar{X}_t)_{t\in\R}$ is stochastically continuous.
The state space of $X$ is $\R$. Let $H\colon\R\to (-1,1)$ be a bi-measurable bijection which is globally Lipschitz continuous, \textit{e.g.} a multiple of the $\mathrm{arctan}$ function. We set $Y_t=H(X_t)$, $\bar{Y}_t=H(\bar{X}_t)$. Let $X_0$, $X_0^\prime$ be some random initial data. We consider the synchronous coupling $(\hat{X}_t,\hat{X}^\prime_t)$, which consists in letting two instances of the process defined by \eqref{SDEalpha} be driven by the same Brownian motion, while starting from $(X_0,X^\prime_0)$. Our aim is to establish an estimate analogous to \eqref{mixCoupled}. Since $f$ has the property $(f(x)-f(y))(x-y)\geq|x-y|^2$, we see, by forming the difference $\hat{X}_t-\hat{X}^\prime_t$ that we have the following pathwise convergence property: for $\PP$-almost all $\omega$, for all measurable time $t_1(\omega)$, for all $t\geq t_1(\omega)$, $|\hat{X}_t-\hat{X}^\prime_t|\leq |\hat{X}_{t_1}-\hat{X}^\prime_{t_1}|e^{-(t-t_1)}$. We apply this inequality with $t_1$ given as the the stopping time
\begin{equation}\label{deftau1}
t_1:=\inf\left\{t\geq 0; |\hat{X}_t|+|\hat{X}_t^\prime|\leq R\right\}.
\end{equation}
The parameter $R$ will be fixed later on. We obtain $\E\left[\mathbf{1}_{t_1\leq t/2}|\hat{X}_t-\hat{X}^\prime_t|\right]\leq R e^{-t/2}$. To get an estimate on the other part $\E\left[\mathbf{1}_{t_1>t/2}|\hat{X}_t-\hat{X}^\prime_t|\right]$, we use the Cauchy-Schwarz inequality. Up to a numerical constant, this gives a bound by the quantity $\PP(t_1>t/2)^{1/2}\left[\E|\hat{X}_t|^2+\E|\hat{X}^\prime_t|^2\right]^{1/2}$.
By the union bound, we have $\PP(t_1>t/2)\leq\PP(|\hat{X}_{t/2}|> R/2)+\PP(|\hat{X}^\prime_{t/2}|> R/2)$, which, using the Markov inequality, we estimate from above by $4R^{-2}\left(\E|X_{t/2}^*|^2+\E|X^{*,\prime}_{t/2}|^2\right)$.
Eventually, we conclude to the following estimate:
\begin{equation}\label{coupling3}
\E\left[|\hat{X}_t-\hat{X}^\prime_t|\right]\leq Re^{-t/2}+8 R^{-2}\left[\sup_{s\geq t/2}\E|\hat{X}_s|^2+\sup_{s\geq t/2}\E|\hat{X}^\prime_s|^2\right]^{3/2}.
\end{equation}
Admit for the moment the following statement.
\begin{proposition} There exists $\kappa>0$ depending on $\alpha$ such that every solution $(X_t)$ to \eqref{SDEalpha} satisfies
\begin{equation}\label{boundSDEalpha}
\E\left[|X_t|^2\right]\leq\kappa\max(1,t^{-1/(1+\alpha)}),
\end{equation}
for all $t>0$.
\label{prop:boundSDEalpha}\end{proposition}
The bound \eqref{boundSDEalpha} is uniform with respect to the starting point $X_0$. For $t\geq 2$, this gives in \eqref{coupling3} the estimate $
\E\left[|\hat{X}_t-\hat{X}^\prime_t|\right]\leq Re^{-t/2}+CR^{-2}$,
where $C$ is a constant depending on $\alpha$. Optimizing in the parameter $R$, we deduce finally that
\begin{equation}\label{egcoupling}
\E\left[|\hat{X}_t-\hat{X}^\prime_t|\right]\leq Ce^{-t/3},
\end{equation}
for all $t\geq 2$, where $C$ is possibly a different constant depending also on $\alpha$. Since $H$ is Lipschitz continuous, the estimate \eqref{egcoupling} can be transferred to $Y$: the coupling $\hat{Y}_t=H(\hat{X}_t)$, $\hat{Y}^\prime_t=H(\hat{X}^\prime_t)$, where $X_0=H^{-1}(Y_0)$, $X_0^\prime=H^{-1}(Y_0^\prime)$, satisfies $\E|\hat{Y}_t-\hat{Y}^\prime_t|\leq C\mathrm{Lip}(H)e^{-t/3}$, for all $t\geq 2$. This gives an instance of \eqref{mixCoupled}. There remains to examine the condition~\eqref{AR0}. Let $U_0^Y$ denote the resolvent operator
\[
U_0^Y\varphi(y)=\int_0^\infty P^Y_t\varphi(y)dt.
\]
The decomposition \eqref{mtdiffusion} gives 
\[
\Lambda\circ U_0\circ\chi(m)=\sum_{j=1}^N a_j \left[U_0^{Y^j}\mathrm{Id}(y^j)\right] \Lambda\circ\chi(n_j),
\]
where $\mathrm{Id}$ is the identity of $(-1,1)$. To get \eqref{AR0}, it is sufficient to ensure that $|U_0^Y\mathrm{Id}|^2\in D(\LL^Y)$. By the correspondences set up in Appendix~\ref{sec:app4}, this is equivalent to the fact that $|U_0^X\mathrm{Id}|^2\in D(\LL^X)$, where $\mathrm{Id}$ now denotes the identity map on $\R$. This can be easily checked: since $e^{-2F}$ is an integrating factor for $\LL^X\colon\varphi\mapsto -f\varphi'+\frac12\varphi''$,
we have the explicit expression
\begin{equation}\label{U0X}
(U^X_0\psi)'(x)=-2\int_x^\infty\psi(y)e^{2(F(x)-F(y)}dy,
\end{equation}
which shows that, when $\psi(y)=y$ (or any odd polynomial), $U^X_0\psi$ is $C^2(\R)$.
To complete the des\-crip\-tion of this example based on diffusion processes, we still have to give the proof of Proposition~\ref{prop:boundSDEalpha}.

\begin{proof}[Proof of Proposition~\ref{prop:boundSDEalpha}] By the It\^o formula, we have
\[
\frac{d\;}{dt}\E\left[|X_t|^2\right]\leq -2\E\left[f(X_t)X_t\right]+1,\quad t>0.
\]
Since $x\mapsto f(x)x$ is convex, we can use the Jensen inequality to obtain the differential inequality
\[
\varphi'(t)\leq g(\varphi(t)),\quad\varphi(t)=\E\left[|X_t|^2\right],\quad g(\varphi)=-2f(\varphi)\varphi+1.
\]
For the value $B=(2(1+\alpha))^{-1/(1+\alpha)}$, the function $t\mapsto\max(1,Bt^{-1/(1+\alpha)})$ is a super-solution to the differential equation $\psi'(t)=g(\psi(t))$. By comparison, starting from an initial $t_0$ small enough, we obtain the desired bound \eqref{boundSDEalpha}.
\end{proof}

\subsection{Main result}\label{sec:mainresult}

\begin{theorem}\label{th:mainth} Assume that $(\bar{\mathtt{m}}_t)$ is an admissible pilot process in the sense of De\-fi\-ni\-tion~\ref{def:admbarm}. Assume that $(M,V,\nu)$ satisfy \eqref{HypM}, \eqref{vdv}, \eqref{nuND}. Let $f^\eps_\mathrm{in}\in L^2(\T^d\times V)$ be a sequence of \textit{non-negative} functions and let $\rho^\eps_\mathrm{in}(x)=\int_V f^\eps_\mathrm{in}(x,v)d\nu(v)$. Suppose also that 
\begin{equation}\label{Hypin}
\rho^\eps_\mathrm{in}\to\rho_\mathrm{in}\text{ in }L^2(\T^d),\quad\sup_{0<\eps<1}\|f^\eps_\mathrm{in}\|_{L^2(\T^d\times V)}\leq C_\mathrm{in}<+\infty,
\end{equation}
and that $\rho_\mathrm{in}\in C^{2+\delta_\mathrm{in}}(\T^d)$ for some positive $\delta_\mathrm{in}$.
Let $f^\eps_t$ be the solution to \eqref{eq:1} with initial datum $f_\mathrm{in}^\eps$ and let $\rho^\eps_t$ be the density associated to $f^\eps_t$ by \eqref{def:rhof}. Let $\bar{M}^\eps_t=\bar{M}_{\eps^{-2}t}$, where $\bar{M}_t$ is the equilibrium given by \eqref{def:barMt}.
Then we have 
\begin{equation}\label{onLocalEquilibrium}
\int_0^t \|f^\eps_s-\rho^\eps_s\bar{M}^\eps_s\|_{L^2(\T^d\times V)}^2ds\leq\frac{C_\mathrm{in}^2 e^t}{\alpha^2}\eps^2,
\end{equation}
almost surely, and, for all $\sigma>0$, the convergence $\rho^\eps_t\to\rho^0_t$ in law on $C([0,T];H^{-\sigma}(\T^d))$, where $\rho^0$ is the solution, in the sense given in Section~\ref{sec:limeq}, to the stochastic partial differential equation
\begin{equation}\label{eq:limeq}
d\rho^0=\divv(K^*\nabla_x\rho^0+\Psi\rho^0)dt+\sqrt{2}\divv_x(\rho^0 S^{1/2}dW(t)),
\end{equation}
with initial condition $\rho^0(0)=\rho_\mathrm{in}$. In \eqref{eq:limeq}, $W(t)$ is a cylindrical Wiener process on $L^2(\T^d)$, $S$ is the covariance operator defined by \eqref{SKernelK}. The coefficients $K^*$ and $\Psi$ have the following expression:
\begin{equation}\label{Kflat}
K^*=K(M)+\E\left[(U_0U_1\chi)(\bar{\mathtt{m}}_0)\otimes\chi(\bar{\mathtt{m}}_0) \right],\quad K(M):=\int_V v\otimes v M(v) d\nu(v),
\end{equation}
and
\begin{equation}\label{coeffPsi}
\Psi=\E\left[\divv_x[\chi(\bar{\mathtt{m}}_0)](U_0U_1\chi)(\bar{\mathtt{m}}_0)\right],
\end{equation}
where $\chi(n)$ and the resolvent $U_\alpha$ are defined in \eqref{defchin} and \eqref{resolventA} respectively.
\end{theorem}
We do several remarks about the content of Theorem	~\ref{th:mainth}.

\begin{remark}[Solutions to \eqref{eq:limeq}] The resolution of \eqref{eq:limeq} is explained in Section~\ref{sec:limeqSolve}. We work in a framework of classical solution (this is the reason why we require the regularity $C^{2+\delta_\mathrm{in}}(\T^d)$ on $\rho_\mathrm{in}$). See the discussion at the beginning of Section~\ref{sec:limeqSolve}.
\end{remark}

\begin{remark}[Enhanced diffusion] In the deterministic case $\bar{\mathtt{m}}^{\eps}\equiv 0$, $f^\eps$ converges to $\rho M$, where $\rho$ is the solution of the diffusion equation
\begin{equation}\label{eq:eqlimdet}
\partial_t \rho - \mathrm{div}\left( K(M)\nabla_x \rho \right)=0,
\end{equation}
with initial condition $\rho(0)=\rho_\mathrm{in}$, see \cite{DegondGoudonPoupaud00}, for example, for a proof of this result. We prove in Proposition~\ref{prop:Kstar} that, at least when the process $(\bar{\mathtt{m}}_t)$ is reversible, we have $K^*\geq K(M)$, in the sense that the matrix $K^*-K(M)$ is non-negative. If we are only concerned with the convergence of the average $r^\eps:=\E\rho^\eps$, we obtain a result of convergence $r^\eps\to r$ in $C([0,T];H^{-\eta}(\T^d))$, where $r$ is a solution to 
\begin{equation}\label{eq:eqlimdetr}
\partial_t r - \mathrm{div}\left( K^*\nabla_x r+\Psi r\right)=0,
\end{equation}
an equation comparable to \eqref{eq:eqlimdet}, with enhanced diffusion.
\label{rk:enhancedDiff}\end{remark}


\begin{remark}[Diffusion-approximation in PDEs] Diffusion-approximation for PDEs has been studied by Pardoux and Piatnitski \cite{PardouxPiatnitski03}, in the context of stochastic homogenization of parabolic equations, by Marty, De Bouard, Debussche, Gazeau, Tsutsumi \cite{Marty06,DeBouardDebussche10,DebusscheTsutsumi11,DeBouardGazeau12} for Schr\"odinger equations and by Bal, Fouque, Garnier, Papanicolaou, S{\o}lna and their co-authors (see \cite{BalGu15,FouqueGarnierPapanicolaouSolna07,GarnierSolna16} for example) for propagation of waves in random media. In the context of hydrodynamic limits, we refer to \cite{DebusscheVovelle12,DebusscheDeMoorVovelle16,DebusscheVovelle20}. A comparison of our work with the reference \cite{DebusscheVovelle20} is done in the next remark.
\end{remark}

\begin{remark}[Comparison with \cite{DebusscheVovelle20}] In the first two papers \cite{DebusscheVovelle12,DebusscheDeMoorVovelle16} the order of the stochastic perturbation is weaker than here in \eqref{eq:1} and, more precisely, the progression is the following one: in \cite{DebusscheVovelle12}, the perturbed test-function method of \cite{PapanicolaouStroockVaradhan77}, developed in the context of ordinary differential equation is combined with the deterministic hydrodynamic limit. In \cite{DebusscheDeMoorVovelle16}, tools for strong convergence are developed and non-linear equations are treated. In \cite{DebusscheVovelle20}, more singular problems (more singular in the sense that the equilibria of the unperturbed equation are stochastic, not deterministic) are considered, in a linear setting however. Here also we consider a singular situation in a linear setting, a framework which is very close to the one considered in \cite{DebusscheVovelle20}. The noticeable difference with \cite{DebusscheVovelle20} is the fact that the space $V$ of velocity is bounded here, while in \cite{DebusscheVovelle20}, $V$ is the whole space $\R^d$. As a consequence, we are able to establish \eqref{onLocalEquilibrium} by means of a relative entropy estimate. This procedure is not working for the time being for the problem considered in \cite{DebusscheVovelle20}, since the properties of localisation in $v$ of the solution are not sufficiently controlled.  The algebra for the computation of the coefficients of the limit equations are quite different in \cite{DebusscheVovelle20} also. This is why we must assume here that $(\bar{\mathtt{m}}_t)$ is reversible to show enhanced diffusion (\textit{cf.} Remark~\ref{rk:enhancedDiff}). Nevertheless, the techniques used in \cite{DebusscheVovelle20} and in the present paper are very similar. This the reason why, sometimes, we use some facts established in \cite{DebusscheVovelle20}. Let us list them: the proof of the Markov property in Theorem~\ref{th:MarkovPty} uses \cite[Theorem~4.5]{DebusscheVovelle20}; the tightness result Proposition~\ref{prop:tight} is similar to \cite[Proposition~5.11]{DebusscheVovelle20}; the identification of the limit via the L\'evy representation theorem in Corollary~\ref{cor:MartingalePbtilde} and the paragraph that follows uses \cite[Section~5E2]{DebusscheVovelle20}; the estimate in Lemma~\ref{lemma:l2coeff} is similar to \cite[Proposition~5.14]{DebusscheVovelle20}; the regularization procedure used in the uniqueness result of Section~\ref{sec:CL} is given with all details in \cite[Section~5E3]{DebusscheVovelle20}; the material on the martingale problem for Markov processes of Appendix~\ref{sec:app3} is taken from the appendix of \cite{DebusscheVovelle20}. At the same time our presentation improves, in our opinion, some similar parts in \cite{DebusscheVovelle20}, in particular by the preparation done in Section~\ref{sec:subsecgenerator2} to the perturbed test-function method, by the result given in Proposition~\ref{prop:SquareOrderone} on the square of the first-order correction, and by the synthesis done in Appendix~\ref{sec:appendix}. In any case, and like \cite{DebusscheVovelle20}, our work leaves open the question of strong convergence of $(\rho^\eps)$. In particular, we are for the moment not able to prove the convergence in law in a space of functions like $L^2([0,T]\times\T^d)$ for instance. This is a limit to the extension of this present work to non-linear equations.
\end{remark}

\begin{remark}[Model of motion by run-and-tumble] We come back to the discussion around \eqref{checkM}. Evolution of agents by run-and-tumble processes has been described in \cite{OthmerDunbarAlt1988} for instance. More precisely, we refer to Section~3 in \cite{OthmerDunbarAlt1988}, where is derived the evolution equation $\partial_t f(t,x,v)+v\cdot\nabla_x f=\mathcal{R}f$, with
\begin{equation}\label{evolODA88}
\mathcal{R}f(t,x,v)=\lambda\int_V \left[T(t,x,v,v')f(t,x,v')-T(t,x,v',v)f(t,x,v)\right]d\nu(v'),
\end{equation}
in the case where the turning kernel $T(t,x,v,v')$ that gives the probability of a jump from $v'$ to $v$ in the redistribution process is independent of $(t,x)$. Possible choices for the turning kernel are discussed in Section~4 of \cite{OthmerDunbarAlt1988}. We also refer to \cite{OthmerHillen01,ChalubMarkowichPerthameSchmeiser04} for models for chemotaxis and related diffusion limits. Here we consider the case $T(t,x,v,v')=M(v)+v\cdot\nabla_x\bar{\mathtt{m}}_t(x)$, see Eq.(23) \& (27) in \cite{ChalubMarkowichPerthameSchmeiser04} for instance. We see that the term related to $\bar{\mathtt{m}}_t(x)$ in $T(t,x,v,v')$ induces a bias in the redistribution process, making more probable the direction $\nabla_x\bar{\mathtt{m}}_t(x)$. If $\bar{\mathtt{m}}_t(x)$ were the concentration of a chemotactic attractant at time $t$ at point $x$, our choice of the turning kernel would give more weight to velocities that drive the organism under consideration towards zone with higher concentration of chemotactic substance. This is conceivable if the organism is big enough to be sensitive to gradients of the chemotactic substance at its own scale. Standard models involve in general the full material derivative $\partial_t S+v\cdot\nabla_x S$ of the chemo-attractant $S$, \cite{SCBPBS11}.
\label{rk:runandtumble}\end{remark}

The organization of the paper is the following one. In Section~\ref{sec:generator}, we study Equation~\eqref{eq:1} at fixed $\eps$, and various properties of the associated Markov process. In particular, we solve the Poisson equation corresponding to Equation~\eqref{eq:3}. This is applied in Section~\ref{sec:perturbed} to set up a method of perturbed test-function method. This method yields the limit generator that arises when $\eps\to 0$. We analyse this generator in Section~\ref{sec:arrange}. The associated SPDE in solved in Section~\ref{sec:limeq}. In the anterior Section~\ref{sec:tightness} we prove the convergence in law of $(\rho^\eps)$ towards a weak solution of the limit equation~\eqref{eq:limeq}. In the last Section~\ref{sec:CL}, we show that this weak solution is unique and strong. In Appendix~\ref{sec:appendix}, various results on Markov processes are given. Although it increases the length of the paper, this addendum seemed to us necessary, for two very different reasons. First, we wanted to specify our approach to infinitesimal generator of Markov processes. Various approaches to this question are possible indeed, especially in infinite dimension (\textit{cf.} the short discussion in appendix~\ref{sec:app3}). We also wanted to gather here in a synthetic way various, somehow related, results whose application in scattered in the paper and is losing its coherence there.

\section{Infinitesimal generator}\label{sec:generator}

\subsection{Notations}\label{sec:notations}

The three first moments of a function $f\in L^1(V,\nu)$ are denoted by
\begin{equation}\label{moments}
R(f)=\int_V f(v)d\nu(v),\quad J(f)=\int_V v f(v)d\nu(v),\quad K(f)=\int_V v\otimes v f(v)d\nu(v).
\end{equation} 
If $E$ is a Banach space and $I$ an interval in $\R$, we denote by $D(I;E)$ the Skorokhod space of c{\`a}dl{\`a}g functions from $I$ to $E$ (see \cite{BillingsleyBook,JacodShiryaev03}). We denote by $\dual{f}{g}$ the duality product between $f$ and $g$, which may be functions defined on $\T^d$ or $\T^d\times V$, \textit{e.g.} $f\in L^1(\T^d\times V)$, $g\in L^\infty(\T^d\times V)$, or distribution and test function. We will specify the duality at stake if necessary. Given $a,b\in\R^d$, $a\otimes b$ is the $d\times d$ matrix with entries $a_ib_j$. The scalar product of two $d\times d$ matrices $A,B$ with real entries is denoted $A : B$. Using the convention of summation over repeated indexes, this is $A:B=A_{ij}B_{ij}$. Let us also give a comment on the terminology here: we will speak of functions which are Lipschitz continuous on bounded sets. Although more concise, we will not use the expression \textit{locally Lipschitz continuous}, since this is sometimes used to indicate a Lipschitz property on restriction to compact sets only. We will denote by $\Phi_t(x,v)=(x+tv,v)$ the flow associated to the field $(v,0)$. Note that $\Phi_t$ preserves the measure on $\T^d\times V$. 

\subsection{Resolution of the kinetic equation}\label{sec:soleps}

We consider here the resolution of the Cauchy Problem \eqref{eq:1}-\eqref{CI} at fixed $\eps$. We assume $\eps=1$ for simplicity. We will give a pathwise resolution of \eqref{eq:1}-\eqref{CI}. More exactly, we construct a solution map $
((\bar{\mathtt{m}}_t),f_\mathrm{in})\mapsto f^\eps$. Since only $q_t(x):=\nabla_x\bar{\mathtt{m}}_t(x)$ does matter here, we will fix $T>0$ and consider the equation
\begin{equation}\label{eq:eps1}
\partial_t f+v\cdot\nabla_x f=R(f)M-f+ R(f) v\cdot q.
\end{equation}
We will construct in Theorem~\ref{th:CYLB} below a solution map 
\begin{equation}\label{SpacesSolve}
L^1(0,t;C(\T^d;\R^d))\times L^1(\T^d\times V)\to C([0,t];L^1(\T^d\times V)),\quad (q,f_\mathrm{in})\mapsto f(t),
\end{equation}
which is continuous. Let us discuss the choice of the space $L^1(0,T;C(\T^d;\R^d))$ made for the variable $q$ here. The state space of the process $(\bar{\mathtt{m}}_t)$ is the Skorokhod space $D([0,T];F)$. The process $q_t=\nabla_x\bar{\mathtt{m}}_t$ is then an element of $D([0,T];C^{r-1}(\T^d;\R^d))$. When solving \eqref{eq:eps1} in the space of functions $f$ continuous in time with values $L^1(\T^d\times V)$, it is quite natural to take the datum $q$ in $L^1(0,T;C(\T^d;\R^d))$. This is not contradictory, since we have the continuous injections
\begin{equation}\label{DDL1}
D([0,T];C^{r-1}(\T^d;\R^d))\hookrightarrow D([0,T];C(\T^d;\R^d))\hookrightarrow L^1(0,T;C(\T^d;\R^d)).
\end{equation}
Let us justify the second injection in \eqref{DDL1} (the first one being obvious). Since $D([0,T];C(\T^d;\R^d))$ is metrizable by \cite[Proposition 1.6]{Jakubowski86}, we consider a sequence $(q_n)$ converging to $q$ in the space $D([0,T];C(\T^d;\R^d))$. This means that there is a sequence $(\lambda_n)$ of strictly increasing maps $\lambda_n\colon[0,T]\to[0,T]$ such that: $\lambda_n(0)=0$, $\lambda_n(T)=T$ for all $n$, and
\begin{equation}\label{cvqn}
\sup_{t\in[0,T]}|t-\lambda_n(t)|\to 0,\quad \sup_{t\in[0,T]}\|q(t)-q_n(\lambda_n(t))\|_{C(\T^d;\R^d)}\to 0,
\end{equation}
when $n\to+\infty$. We deduce from \eqref{cvqn} that
\begin{equation}\label{cvqnINV}
\sup_{t\in[0,T]}|t-\lambda_n^{-1}(t)|\to 0,\quad \sup_{t\in[0,T]}\|q(\lambda_n^{-1}(t))-q_n(t)\|_{C(\T^d;\R^d)}\to 0,
\end{equation}
when $n\to+\infty$. From the triangular inequality
\[
\|q(t)-q_n(t)\|_{C(\T^d;\R^d)}\leq \|q(t)-q(\lambda_n^{-1}(t))\|_{C(\T^d;\R^d)}+\|q(\lambda_n^{-1}(t))-q_n(t)\|_{C(\T^d;\R^d)}
\] 
and \eqref{cvqnINV}, we deduce that $q_n(t)\to q(t)$ in $C(\T^d;\R^d)$ if $q$ is continuous at $t$. The set of points of discontinuity of $q$ is at most countable, \cite[p.264]{Jakubowski86}, therefore $q_n(t)\to q(t)$ for almost all $t\in[0,T]$. By \eqref{cvqnINV}, the quantity $\sup_{t\in[0,T]}\|q_n(t)\|_{C(\T^d;\R^d)}$ is uniformly bounded in $n$. We can apply the Lebesgue dominated convergence theorem to conclude that $q_n\to q$ in $L^1(0,T;C(\T^d;\R^d))$.

\begin{definition} Let $f_\mathrm{in}\in L^1(\T^d\times V)$, let $q\in L^1(0,T;C(\T^d;\R^d))$. Let $\Phi_t(x,v)=(x+tv,v)$. A continuous function from $[0,T]$ to $L^1(\T^d\times V)$ is said to be a mild solution to \eqref{eq:eps1} with initial datum $f_\mathrm{in}$ if
\begin{equation}\label{mildLB}
f(t)=e^{-t}f_\mathrm{in}\circ\Phi_{-t}+\int_0^t e^{-(t-s)}[ R(f(s)) (M+v\cdot q(s,\cdot))]\circ\Phi_{-(t-s)}ds,
\end{equation}
for all $t\in[0,T]$.
\label{def:mildLB}\end{definition}

\begin{theorem} Let $f_\mathrm{in}\in L^1(\T^d\times V)$, let $q\in L^1(0,T;C(\T^d;\R^d))$. There exists a unique mild solution to \eqref{eq:eps1} in $C([0,T];L^1(\T^d\times V))$ with initial datum $f_\mathrm{in}$. It satisfies the bound
\begin{equation}\label{sol:fLB}
\|f(t)\|_{L^1(\T^d\times V)}\leq e^{\int_0^t\|q(s)\|_{C(\T^d;\R^d)}ds}\|f_\mathrm{in}\|_{L^1(\T^d\times V)}\quad\text{for all }t\in [0,T].
\end{equation} 
Two mild solutions $f^1$, $f^2$ associated to two sets of data $(f^j_\mathrm{in},q^j)$, $j=1,2$, satisfy the estimate
\begin{multline}\label{sol:fLB12}
\|f^1(t)-f^2(t)\|_{L^1(\T^d\times V)}
\leq A(t)\left(\|f^1_\mathrm{in}-f^2_\mathrm{in}\|_{L^1(\T^d\times V)}+\int_0^T\|q^1(t)-q^2(t)\|_{C(\T^d;\R^d)}\right),
\end{multline} 
for all $t\in [0,T]$, where the constant $A$ depends on $t$, and on $\|f^j_\mathrm{in}\|_{L^1(\T^d\times V)}$ and $\|q^j\|_{L^1(0,t;C(\T^d;\R^d))}$, for $j=1,2$.
In particular, the map \eqref{SpacesSolve} is Lipschitz continuous on bounded sets.
\label{th:CYLB}\end{theorem}

\begin{proof}[Proof of Theorem~\ref{th:CYLB}] Let $E_T$ denote the  space of continuous functions from $[0,T]$ to $L^1(\T^d\times\R^d)$. We use the norm
$
\|f\|_{E_T}=\sup_{t\in[0,T]}\left(1+\|q(s)\|_{C(\T^d;\R^d)}\right)\|f(t)\|_{L^1(\T^d\times V)}
$
on $E_T$. Note that
\begin{equation}\label{rhotof}
\| R(f)\|_{L^1(\T^d)}\leq\|f\|_{L^1(\T^d\times V)}.
\end{equation}
Let $f\in E_T$. Assume that \eqref{mildLB} is satisfied. Then, by \eqref{rhotof}, and due to the fact that $v\in V$ has a norm less than $1$, we have
\begin{align*}
\|f(t)\|_{L^1(\T^d\times V)}\leq & e^{-t}\|f_\mathrm{in}\|_{L^1(\T^d\times V)}+\int_0^t e^{-(t-s)}(1+\|q(s)\|_{C(\T^d;\R^d)})\|f(s)\|_{L^1(\T^d\times V)}ds.
\end{align*}
By Gr\"onwall's Lemma applied to $t\mapsto e^t\|f(t)\|_{L^1(\T^d\times V)}$, we obtain \eqref{sol:fLB} as an a priori estimate.
Besides, the $L^1$-norm of the integral term in \eqref{mildLB} can be estimated by $(1-e^{-T})\|f\|_{E_T}$. This means that the application which, to $f\in E_t$, associates the right-hand side of \eqref{mildLB}, is a contraction of $E_T$. Existence and uniqueness of a solution to \eqref{mildLB} in $E_T$ follow from the Banach fixed point Theorem. Using the linearity of the equation, \eqref{sol:fLB12} is obtained by similar estimates.
\end{proof}

To complete Theorem~\ref{th:CYLB}, we give the following result.

\begin{proposition}[Non-negative solutions] Let $f_\mathrm{in}\in L^1(\T^d\times V)$, let $q\in L^1(0,T;C(\T^d;\R^d))$. Let $f$ be the unique mild solution to \eqref{eq:eps1} in $C([0,T];L^1(\T^d\times V))$ with initial datum $f_\mathrm{in}$. Assume that
\begin{equation}\label{ModifiedEqPos}
M(v)+v\cdot q(t,x)\geq 0,
\end{equation}
for a.e. $(t,x,v)\in (0,T)\times\T^d\times V$ and that $f_\mathrm{in}\geq 0$ a.e. Then $f\geq 0$ a.e. on $(0,T)\times\T^d\times V$.
\label{prop:posmildLB}\end{proposition}

\begin{proof}[Proof of Proposition~\ref{prop:posmildLB}] In view of \eqref{mildLB}, it is sufficient to show that $ R(f)\geq 0$ a.e. on $(0,T)\times\T^d$. We write $R(f)(t)$ as the convex combination
$$
 R(f)(t)=e^{-t} R(f_\mathrm{in}\circ\Phi_{-t})+(1-e^{-t})\int_0^t\int_V \frac{e^{-(t-s)}}{1-e^{-t}}[ R(f(s)) (M+v\cdot q(s,\cdot))]\circ\Phi_{-(t-s)}ds d\nu.
$$
By convexity of $s\mapsto s^-$, we deduce that
\begin{multline*}
[ R(f)(t)]^-\leq e^{-t}[ R(f_\mathrm{in}\circ\Phi_{-t})]^-\\
+(1-e^{-t})\int_0^t\int_V \frac{e^{-(t-s)}}{1-e^{-t}}[ R(f(s)) (M+v\cdot q(s,\cdot))]^-\circ\Phi_{-(t-s)}ds d\nu.
\end{multline*}
Using \eqref{ModifiedEqPos} and $f_\mathrm{in}\geq 0$, we obtain
\begin{equation}\label{intphiminus}
[ R(f)(t)]^-\leq \int_0^t\int_V e^{-(t-s)}[ R(f(s))^- (M+v\cdot q(s,\cdot))]\circ\Phi_{-(t-s)}ds d\nu
\end{equation}
We integrate \eqref{intphiminus} over $x\in\T^d$. Since $\Phi_t$ is measure preserving, we obtain
$$
e^t\| R(f)^-\|_{L^1(\T^d)}(t)\leq\int_0^t (1+\|q(s)\|_{C(\T^d;\R^d)})e^{s}\| R(f)^-\|_{L^1(\T^d)}(s)ds.
$$
The Gr\"onwall Lemma implies $ R(f)^-=0$ a.e. on $(0,T)\times\T^d$.
\end{proof}

We will also need the following result about the regularity of solutions.

\begin{proposition}[Propagation of regularity] Let $f_\mathrm{in}\in L^1(\T^d\times V)$ satisfy
\begin{equation}\label{regfin}
f_\mathrm{in}\in L^2(\T^d\times V),\quad\partial_{x_i}f_\mathrm{in}\in L^2(\T^d\times V),
\end{equation} 
for $i\in\{1,\ldots,d\}$. Let $q\in L^1(0,T;C^1(\T^d;\R^d))$. Let $f$ be the unique mild solution to \eqref{eq:eps1} in $C([0,T];L^1(\T^d\times V))$ with initial datum $f_\mathrm{in}$. Then $f(t)$ and $\nabla_x f(t)\in L^2(\T^d\times V)$ for all $t\in[0,T]$, and we have the estimate
\begin{equation}\label{regfLB}
\|f(t)\|_{L^2(\T^d\times V)}^2+\sum_{i=1}^d\|\partial_{x_i}f(t)\|_{L^2(\T^d\times V)}^2 \leq D(T)\left[\|f_\mathrm{in}\|_{L^2(\T^d\times V)}^2+\sum_{i=1}^d\|\partial_{x_i}f_\mathrm{in}\|_{L^2(\T^d\times V)}^2\right],
\end{equation}
for all $t\in[0,T]$, where the constant $D(T)$ depends on $T$, on the constant $\alpha$ in \eqref{HypM} and on the norm $\|q\|_{L^1(0,T;C^1(\T^d;\R^d))}$.
\label{prop:regmildLB}\end{proposition}

\begin{proof}[Proof of Proposition~\ref{prop:regmildLB}] The mild solution $f$ to \eqref{eq:eps1} is obtained by a fixed-point argument. Therefore $f$ is the limit, in $C([0,T];L^1(\T^d\times V))$, of the iterative sequence $f^k$ defined by $f^0=f_\mathrm{in}$, $f^{k+1}$ solution to the equation
\begin{equation}\label{eq:k}
\partial_t f^{k+1}+v\cdot\nabla_x f^{k+1}= R(f^k) [M(v)+v\cdot q]-f^{k+1},
\end{equation}
with initial condition $f_\mathrm{in}$, in the sense that 
\begin{equation}\label{mildLBk}
f^{k+1}(t)=e^{-t}f_\mathrm{in}\circ\Phi_{-t}+\int_0^t e^{-(t-s)}[ R(f^k(s)) (M+v\cdot q(s,\cdot))]\circ\Phi_{-(t-s)}ds,
\end{equation}
for all $t\in[0,T]$. On the basis of \eqref{mildLBk}, using \eqref{HypM}, we derive the $L^2$-bound
\begin{align*}
\|f^{k+1}(t)\|_{L^2(\T^d\times V)}&
\leq e^{-t}\|f_\mathrm{in}\|_{L^2(\T^d\times V)}
+\int_0^t e^{-(t-s)}\left(\alpha^{-1}+\|q(s)\|_{C(\T^d;\R^d)}\right)\| R(f^k(s))\|_{L^2(\T^d)}ds.
\end{align*}
Since $\| R(f)\|_{L^2(\T^d)}\leq\|f\|_{L^2(\T^d\times V)}$ by Jensen's Inequality, we obtain
$$
\varphi_{k+1}(t)\leq \|f_\mathrm{in}\|_{L^2(\T^d\times V)}+\int_0^t\left(\alpha^{-1}+\|q(s)\|_{C(\T^d;\R^d)}\right)\varphi_k(s)ds,\quad\varphi_k(t):=e^t\|f^{k}(t)\|_{L^2(\T^d\times V)}.
$$
We conclude that
$$
\|f^{k}(t)\|_{L^2(\T^d\times V)}\leq \exp\left[\int_0^t\left(\alpha^{-1}-1+\|q(s)\|_{C(\T^d;\R^d)}\right)ds\right]\|f_\mathrm{in}\|_{L^2(\T^d\times V)},
$$
which gives a similar bound for $f$ at the limit $k\to+\infty$. The bound on the derivatives with respect to $x$ of $f$ is obtained similarly on the basis of \eqref{mildLBk}, by differentiation and $L^2$-estimate as above. We conclude to \eqref{regfLB}.
\end{proof}

\subsection{Infinitesimal generator}\label{sec:subsecgenerator}

We emphasized the fact that the map \eqref{SpacesSolve} is continuous in the statement of Theorem~\ref{th:CYLB}, because this property is used in the proof of the Markov property for $(f_t,\mathtt{m}(t;m))_{t\geq 0}$. We will not give the details of the proof however; we refer to Theorem~4.5 in \cite{DebusscheVovelle20} instead.

\begin{theorem}[Markov property] Let $(\bar{\mathtt{m}}_t)$ be an admissible pilot process in the sense of Definition~\ref{def:admbarm}. Let $\mathcal{X}$ denote the state space
\begin{equation}\label{defXX}
\mathcal{X}=L^1(\T^d\times V)\times\mathtt{F}.
\end{equation}
For $(f,m)\in\mathcal{X}$, let $f_t$ denote the mild solution to \eqref{eq:eps1} with initial datum $f$ and forcing $q_t=\nabla_x \mathtt{m}(t;m)$. Then $(f_t,\mathtt{m}(t;m))_{t\geq 0}$ is a time-homogeneous Markov process over $\mathcal{X}$. It is locally bounded in the sense of Remark~\ref{rk:item-locbound} and it is stochastically continuous.
\label{th:MarkovPty}\end{theorem}
\begin{proof}[Proof of Theorem~\ref{th:MarkovPty}] We will give only few details of the proof. The proof of the Markov property uses the description of $f_t$ as a continuous flow by means of \eqref{SpacesSolve}. We refer to Theorem~4.5 in \cite{DebusscheVovelle20}. The fact that $(f_t,\mathtt{m}_t)$ is locally bounded follows from \eqref{BallR} and the bound~\eqref{sol:fLB} which give: almost surely, for all $t\in [0,T]$,
\begin{equation}\label{boundthMarkovPty}
\|f_t\|_{L^1(\T^d\times V)}\leq e^{\mathtt{b}T}\|f_0\|_{L^1(\T^d\times V)}.
\end{equation} 
Since $(\mathtt{m}_t)$ is stochastically continuous and $(f_t)$ is almost surely continuous, the joint process $(f_t,\mathtt{m}_t)$ is stochastically continuous.
\end{proof}

We denote by $(\mathscr{P}_t)$ the semi-group associated to $(f_t,\mathtt{m}_t)_{t\geq 0}$:
\begin{equation}\label{defPt}
\mathscr{P}_t\varphi(f,m):=\E\varphi(f_t,\mathtt{m}(t;m)),\quad\varphi\in \mathrm{BM}(\mathcal{X}).
\end{equation}
Coming back to the case $\eps>0$ (instead of $\eps=1$), we obtain that, representing by $f^\eps_t$ the mild solution to \eqref{eq:1}, the process $(f^\eps_t,\bar{\mathtt{m}}_{\eps^{-2}t})$ is a Markov process. We will denote by $(\mathscr{P}^\eps_t)$ the corresponding semi-group. Formally, the infinitesimal generator $\LL^\eps$ associated to $(\mathscr{P}^\eps_t)$ is given as
$
\LL^\eps=\eps^{-2}\LL_\sharp+\eps^{-1}\LL_\flat,
$
where $\LL_\sharp$ and $\LL_\flat$ are defined by
\begin{align}
\LL_\sharp\varphi(f,m)=&\mathtt{A}\varphi(f,m)+(M\rho-f+ R(f)v\cdot\nabla_x m,D_f\varphi(f,m)),\label{Lsharp}
\end{align}
and
\begin{align}
\LL_\flat\varphi(f,m)=&-(v\cdot\nabla_x f,D_f\varphi(f,m)),\label{Lflat}
\end{align}
for $(f,m)\in\mathcal{X}$. 
In what follows, we study the main generator $\LL_\sharp$ (Section~\ref{sec:main-generator}). Then, in Section~\ref{sec:subsecgenerator2}, we also specify the meaning of $\LL^\flat$ and we establish that a certain class of functions is in the domain of $\LL^\eps$. These results will be applied then to construct some adequate perturbed test functions in Section~\ref{sec:perturbed}.

\subsection{Main generator}\label{sec:main-generator}

The main generator $\LL_\sharp$, defined in \eqref{Lsharp}, is associated to the following equation (started at $t=t_0$)
\begin{equation}
\begin{cases}\label{sode}
\frac{d}{dt}g_t = R(g_t)M-g_t +  R(g_t) v\cdot \nabla_x \mathtt{m}(t,t_0;m),\\
g_{t_0} = g.
\end{cases}
\end{equation}
A solution to \eqref{sode} satisfies (formally in a first step) $ R(g_t)= R(g)$. Therefore \eqref{sode} is a simple dissipative equation on $g_t$, with source term. The explicit solution to \eqref{sode} reads
\begin{equation}\label{explicitgt}
g_{t_0,t} = e^{-\left(t-t_0\right)}g+ R(g)M\left(1-e^{-(t-t_0)}\right)+ R(g)v\cdot\nabla_x\int_{t_0}^{t}e^{-(t-s)}\mathtt{m}(s,t_0;m) ds.
\end{equation}

The perturbed test-function method set out in Section~\ref{sec:perturbed} involves the resolution of the Poisson equation for $\LL_\sharp$. In Appendix~\ref{sec:app5}, we give the details of the resolution of the Poisson equation for a certain class of Markov processes. To check that the process $(g_t,\mathtt{m}_t)$ belongs to this class of processes, and also to facilitate the definition of the correctors in Section~\ref{sec:perturbed}, it will be easier to consider the following extended process $(g_t,\mathtt{w}_t,\mathtt{m}_t)$, described by the extended semi-group $(\tilde{P}^\sharp_t)$ given as follows: for $(g,w,m)$ in the state space 
\begin{equation}\label{defXtilde}	
\tilde{\mathcal{X}}=L^1(\T^d\times V)\times\mathtt{F}\times\mathtt{F},
\end{equation} 
set
\begin{equation}\label{gwm}
\mathtt{w}_t=e^{-t}w+\int_{0}^{t}e^{-(t-s)}\mathtt{m}(s,0;m)ds,\quad g_t=e^{-t}g+(1-e^{-t})R(g)M+R(g) v\cdot \nabla_x(\mathtt{w}_t-e^{-t}w),
\end{equation}
and $\tilde{P}^\sharp_t\phi(g,w,m)=\E\phi(g_t,\mathtt{w}_t,\mathtt{m}(t;m))$.
Let $\rho$ be an arbitrary element of $L^1(\T^d)$. Since $R(g_t)=R(g)$, the Markov process $(g_t,\mathtt{w}_t,\mathtt{m}_t)$ can be restricted to the state space
\begin{equation}\label{Xrho}
\tilde{\mathcal{X}}_\rho=\left\{g\in L^1(\T^d\times V);R(g)=\rho\right\}\times\mathtt{F}\times\mathtt{F}.
\end{equation}
Since $\mathtt{F}$ is convex by hypothesis, we have indeed $\mathtt{w}_t\in\mathtt{F}$ when $w\in\mathtt{F}$. On $\tilde{\mathcal{X}}$, we consider the norm given by
\[
\|(g,w,m)\|_{\tilde{\mathcal{X}}}=\|g\|_{L^1(\T^d\times V)}+\|w\|_F+\|m\|_F.
\]
In the following proposition, we use the terminology of Appendix~\ref{sec:app5}.
\begin{proposition} Let $\rho\in L^1(\T^d)$. The Markov process $(g_t,\mathtt{w}_t,\mathtt{m}_t)$ with state space $\tilde{\mathcal{X}}_\rho$ is locally bounded uniformly in time. It has the property of stochastic continuity described in Item~\ref{item-stocont} of Section~\ref{sec:app3}. It has a unique invariant measure $\tilde{\mu}_\rho$ which has a bounded support and satisfies a mixing hypothesis as in \eqref{coupling-appendix}.
\label{prop:prePoissonqwm}\end{proposition}

\begin{proof}[Proof of Proposition~\ref{prop:prePoissonqwm}] By \eqref{BallR} and \eqref{gwm} and some trivial estimates, we have, $\PP$-almost surely,
\[
\|(g_t,\mathtt{w}_t,\mathtt{m}_t)\|_{\tilde{\mathcal{X}}}\leq \|g\|_{L^1(\T^d\times V)}+\|\rho\|_{L^1(\T^d)}(1+\mathtt{b})+2\mathtt{b}.
\]
This shows the local bound, uniformly in time. Since $(g_t,\mathtt{w}_t)$ is almost surely continuous and $(\mathtt{m}_t)$ is stochastically continuous, the process $(g_t,\mathtt{w}_t,\mathtt{m}_t)$ is stochastically continuous. The associated semi-group $(\tilde{P}^\sharp_t)$ has therefore the desired stochastic continuity property. Recall that $\bar{M}_t$ and $\bar{\mathtt{w}}_t$ are given in \eqref{def:barMt}. We define the measure $\tilde{\mu}_\rho$ as the law of the process $(\rho\bar{M}_t,\bar{\mathtt{w}}_t,\bar{\mathtt{m}}_t)$:
\begin{equation}\label{deftildemu}
\dual{\tilde{\mu}_\rho}{\Phi}=\E\Phi(\rho\bar{M}_t,\bar{\mathtt{w}}_t,\bar{\mathtt{m}}_t).
\end{equation}
Since $(\bar{\mathtt{m}}_t)$ is stationary, this definition is independent of $t$. The trivial bound $\|(\rho\bar{M}_t,\bar{\mathtt{w}}_t,\bar{\mathtt{m}}_t)\|_{\tilde{\mathcal{X}}_\rho}\leq 1+3\mathtt{b}$ shows that the measure $\tilde{\mu}_\rho$ has a bounded support. When the starting point is $(g,w,m)=(\rho\bar{M}_0,\bar{\mathtt{w}}_0,\bar{\mathtt{m}}_0)$, we have, for $s\geq 0$, $\mathtt{m}(s,0;\bar{\mathtt{m}}_0)=\bar{\mathtt{m}}_s$ almost surely and we compute, using \eqref{gwm}, $\mathtt{w}_t=\bar{\mathtt{w}}_t$, $g_t=\rho\bar{M}_t$. Therefore $\tilde{\mu}_\rho$ is an invariant measure. The fact that it is unique and ergodic will follow from the mixing estimate that we will now establish. Let $(g,w,m)$ and $(g^\prime,w^\prime,m^\prime)$ be given in $\tilde{\mathcal{X}}_\rho$. Let $(\hat{\mathtt{m}}_t,\hat{\mathtt{m}}^\prime_t)$ be a coupling of $(\mathtt{m}(t;m),\mathtt{m}(t;m^\prime))$ such that $\E\|\hat{\mathtt{m}}_t-\hat{\mathtt{m}}^\prime_t\|_F\leq\mathtt{b}\gamma_\mathrm{mix}(t)$ (the coupling is given by Hypothesis~\ref{mixCoupled}). We use the definition~\eqref{gwm} with ``hats'' and ``hats and primes'' to define the processes $\hat{X}_t:=(\hat{g}_t,\hat{\mathtt{w}}_t,\hat{\mathtt{m}}_t)$ and $\hat{X}_t^\prime:=(\hat{g}_t^\prime,\hat{\mathtt{w}}_t^\prime,\hat{\mathtt{m}}_t^\prime)$. Let us set
\[
\gamma^*_\mathrm{mix}(t)=\int_0^t e^{-(t-s)}\gamma_\mathrm{mix}(s)ds.
\]
Some simple estimates then give
\begin{equation}\label{coupling-diff0}
\|\hat{X}_t\|_{\tilde{\mathcal{X}}}\leq C,\quad 
\|\hat{X}_t^\prime\|_{\tilde{\mathcal{X}}}\leq C,\quad
\E\left[\|\hat{X}_t-\hat{X}_t^\prime\|_{\tilde{\mathcal{X}}}\right]\leq C\Gamma_\mathrm{mix}(t),
\end{equation}
where $\Gamma_\mathrm{mix}(t)=e^{-t}+\gamma^*_\mathrm{mix}(t)+\gamma_\mathrm{mix}(t)$ is a continuous function in $L^1(\R_+)$. The constant $C$ in \eqref{coupling-diff0} depends on the size of the data in $\tilde{\mathcal{X}}$. Indeed, by inspection in \eqref{gwm} we obtain the first two bounds in \eqref{coupling-diff0} and also the last estimate in \eqref{coupling-diff0} (which corresponds to  \eqref{coupling-appendix}), where the constant $C$ depends on $\|g\|_{L^1(\T^d\times V)}$, $\|g^\prime\|_{L^1(\T^d\times V)}$ and $\mathtt{b}$. 
\end{proof}

\begin{corollary} Let $\varphi\colon\tilde{\mathcal{X}}\to\R$ be Lipschitz continuous on bounded sets. Then the function
\begin{equation}\label{GammaMixsharpLipf}
\psi(g,w,m):=-\int_0^\infty (\tilde{P}^\sharp_t\varphi(g,w,m)-\dual{\tilde{\mu}_{R(g)}}{\varphi}) dt,\quad (g,w,m)\in\tilde{\mathcal{X}},
\end{equation}
is Lipschitz continuous on bounded sets with respect to the variable $g$. 
\label{cor:prop:prePoissonqwm}\end{corollary}

\begin{proof}[Proof of Corollary~\ref{cor:prop:prePoissonqwm}] We use the notations of the proof of Proposition~\ref{prop:prePoissonqwm}. Let us fix $X=(g,w,m)\in\tilde{\mathcal{X}}$ and let us take $X^\prime=(g^\prime,w,m)$. We have then
\[
\psi(g,w,m)-\psi(g^\prime,w,m)=-\int_0^\infty\E\left[\varphi(\hat{X}_t)-\varphi(\hat{X}_t^\prime)\right]dt.
\]
We use \eqref{coupling-diff0} and the fact that $\varphi$ is Lipschitz continuous on bounded sets of $\tilde{\mathcal{X}}$ to get an estimate
\[
|\psi(g,w,m)-\psi(g^\prime,w,m)|\leq C\|g-g^\prime\|_{L^1(\T^d\times V)},
\]
where the constant $C$ depends on $r:=\|(g,w,m)\|_{\tilde{\mathcal{X}}}+\|g^\prime\|_{L^1(\T^d\times V)}$, on the Lipschitz constant of $\varphi$ on the ball $\bar{B}(0,r)$ of $\tilde{\mathcal{X}}$ and on the $L^1(\R_+)$-norm of $\Gamma_\mathrm{mix}$.
\end{proof}
\subsection{Infinitesimal generator: some elements in the domain}\label{sec:subsecgenerator2}

In the following proposition, we use the results on the main generator given in the previous Section~\ref{sec:main-generator}. For $\rho\in L^1(\T^d)$, the state space $\tilde{\mathcal{X}}_\rho$ is defined by \eqref{Xrho}. The state spaces $\mathcal{X}$, $\tilde{\mathcal{X}}$ are defined by \eqref{defXX}, \eqref{defXtilde} respectively. We also denote by $\mathcal{X}_\rho$ the image of $\tilde{\mathcal{X}}_\rho$ by the projection on the first and third coordinates:
\[
\mathcal{X}_\rho\subset\mathcal{X},\quad \mathcal{X}_\rho=\left\{g\in L^1(\T^d\times V);R(g)=\rho\right\}\times\mathtt{F}.
\]
We will also slightly abuse notations by considering any function $\Phi$ of the variables $(g,m)$ as a function of $(g,w,m)$. This is the trivial way to lift functions defined on $\mathcal{X}$ to functions defined on $\tilde{\mathcal{X}}$. One can check on \eqref{gwm} that $\tilde{P}^\sharp_t\Phi$ is then also a function independent of $w$. We have studied the generator $\LL^\sharp$ in Section~\ref{sec:main-generator}. The generator $\LL^\flat$ is understood as the infinitesimal generator, according to the definition of the appendix~\ref{sec:app2}, associated to the semi group $(P^\flat_t)$ given by
\begin{equation}\label{Pflat}
P^\flat_t\varphi(f,m)=\varphi(f\circ\Phi_{-t},m),\quad\Phi_t(x,v)=(x+tv,v).
\end{equation}

In the following proposition, we will exhibit some functions solutions to the Poisson equation $\LL^\sharp\psi(f,m)=\varphi(f,m)-\dual{\mu_{R(f)}}{\varphi}$. The test functions $\varphi\colon\mathcal{X}\to\R$ will be restricted to the following class of \emph{admissible} and \emph{good test functions}.

\begin{definition}[Admissible and good test functions] Let $\varphi\colon\mathcal{X}\to\R$ be Lipschitz continuous on bounded sets. We say that $\varphi$ is a \emph{good test function} if $\varphi$ is differentiable with respect to $f$ at all points, with $D_f\varphi$ Lipschitz continuous on bounded sets, and if the following regularizing conditions are satisfied: for all $i=1,\dotsc,d$, for all $g\in C^1(\T^d;L^1(V))$, for all $(f,m)\in\mathcal{X}$,
\begin{equation}\label{regxvarphi}
|(\partial_{x_i}g,D_f\varphi(f,m))|\leq C_1(\|(f,m)\|_{\mathcal{X}})\|g\|_{L^1(\T^d\times V)},
\end{equation}
where $C_1$ is locally bounded on $\R_+$, and the following additional Lipschitz dependence in \eqref{regxvarphi} is satisfied:
\begin{multline}\label{regxvarphiLip}
|(\partial_{x_i}g,D_f\varphi(f,m)-D_f\varphi(f^\prime,m^\prime))|\\
\leq C_1(\|(f,m)\|_{\mathcal{X}}+\|(f^\prime,m^\prime)\|_{\mathcal{X}})\|g\|_{L^1(\T^d\times V)}
\|(f,m)-(f^\prime,m^\prime)\|_{\mathcal{X}},
\end{multline}
for every $(f^\prime,m^\prime)\in\mathcal{X}$. If, save for the other properties, \eqref{regxvarphiLip} is satisfied under the restriction that $m=m^\prime$ (the Lipschitz dependence in \eqref{regxvarphi} being satisfied with respect to $f$ only hence), we say that $\varphi$ is an \emph{admissible test function}.
\label{def:ADMtestfunctions}\end{definition}

\begin{proposition} Let $(\bar{\mathtt{m}}_t)$ be an admissible pilot process in the sense of Definition~\ref{def:admbarm}. Let $(\tilde{P}^\sharp_t)$, \textit{resp.} $(P^\flat_t)$, be the semi-group defined in Section~\ref{sec:main-generator}, \textit{resp.} by \eqref{Pflat}. Let $\varphi\colon\mathcal{X}\to\R$ be either an admissible or good test-function. We set
\begin{equation}\label{correctorphipsi}
\psi(f,m)=-\int_0^\infty\left[ \tilde{P}^\sharp_t\varphi(f,m)-\dual{\tilde{\mu}_{R(f)}}{\varphi}\right]dt,\quad (f,m)\in\mathcal{X}.
\end{equation}
We have then the following results:
\begin{enumerate}
\item\label{item-varphiLflat} if $\varphi$ is an admissible test function, then $\varphi$ belongs to the domain of $\LL^\flat$, and $\LL^\flat\varphi$ is Lipschitz continuous on bounded sets,
\item\label{item-varphipsiOK} if $\varphi$ is a good test function, then the function $\psi$ is well defined and is an admissible test function,
\item\label{item-psiLL} if $\varphi$ is a good test function, then $\psi$ is in the intersection of the three domains $D(\LL^\sharp)$, $D(\LL^\flat)$, $D(\LL^\eps)$ and satisfies the Poisson equation $\LL^\sharp\psi=\varphi$ in $\mathcal{X}$.
\end{enumerate}
\label{prop:admtest1}\end{proposition}

\begin{proof}[Proof of Proposition~\ref{prop:admtest1}] We first assume that $\varphi$ is an admissible test function. We want to show that $\varphi\in D(\LL^\flat)$. Although there is no more than a standard chain-rule for differentiation behind this result, we will justify it carefully, in the framework for infinitesimal generator based on $\pi$-convergence (with extension to functions bounded on bounded sets for locally bounded Markov processes), as described in Appendix~\ref{sec:app2}. First, by integration, \eqref{regxvarphi} gives the estimates
\begin{equation}\label{regxvarphi-int}
|t^{-1}(g\circ\Phi_{-t}-g,D_f\varphi(f,m))|\leq C_1(\|(f,m)\|_{\mathcal{X}})\|g\|_{L^1(\T^d\times V)},
\end{equation}
and
\begin{equation}\label{regxvarphi-int2}
|(t^{-1}[g\circ\Phi_{-t}-g]-v\cdot\nabla_x g,D_f\varphi(f,m))|\leq C_1(\|(f,m)\|_{\mathcal{X}})\sup_{0\leq\sigma\leq t}\|g\circ\Phi_\sigma-g\|_{L^1(\T^d\times V)}.
\end{equation}
To establish the second estimate \eqref{regxvarphi-int2} for instance, we write
\[
t^{-1}[g\circ\Phi_{-t}-g]-v\cdot\nabla_x g=\int_0^1 v\cdot\nabla_x(g\circ\Phi_{-st}-g)ds
\]
and apply \eqref{regxvarphi} to $v_i(g\circ\Phi_{st}-g)$. Integration in \eqref{regxvarphiLip} also leads to the estimate
\begin{multline}\label{regxvarphiLip-int}
|(t^{-1}[g\circ\Phi_{-t}-g],D_f\varphi(f,m)-D_f\varphi(f^\prime,m^\prime))|\\
\leq C_1(\|(f,m)\|_{\mathcal{X}}+\|(f^\prime,m^\prime)\|_{\mathcal{X}})\|g\|_{L^1(\T^d\times V)}
\|(f,m)-(f^\prime,m^\prime)\|_{\mathcal{X}},
\end{multline}
All five estimates \eqref{regxvarphi}, \eqref{regxvarphiLip}, \eqref{regxvarphi-int}, \eqref{regxvarphi-int2}, \eqref{regxvarphiLip-int} can be extended by density to functions $g\in L^1(\T^d\times V)$. We combine these estimates with the expansion
\[
\varphi(f+th,m)-\varphi(f,m)=t\int_0^1 (h,D_f\varphi(f+sth,m)ds
\]
where $h:=t^{-1}[f\circ\Phi_{-t}-f]$ to get the desired properties of the expansion $\varphi(f\circ\Phi_{-t},m)-\varphi(f,m)$. Indeed, we obtain the chain rule
$\varphi(f\circ\Phi_{-t},m)-\varphi(f,m)=t(v\cdot\nabla_x f,D_f\varphi(f,m))
+t\eta_t(f,m)$, where
\begin{multline}\label{chainruleCorrector}
\eta_t(f,m)=(t^{-1}[f\circ\Phi_{-t}-f]-v\cdot\nabla_x f,D_f\varphi(f,m))\\
+\int_0^1 (t^{-1}[f\circ\Phi_{-t}-f],D_f\varphi(f+s(f\circ\Phi_{-t}-f),m)-D_f\varphi(f,m))ds.
\end{multline}
Using \eqref{regxvarphi-int2} and \eqref{regxvarphiLip-int}, we see that $\eta_t$ is bounded on bounded sets and satisfies the pointwise convergence $\eta_t(f,m)\to 0$ when $t\to 0$, for every $(f,m)\in\mathcal{X}$. Additionally, the Lipschitz property \eqref{regxvarphiLip} is inherited by $\LL^\flat\varphi$, which is Lipschitz continuous on bounded sets. 
To establish Item~\ref{item-varphipsiOK}, we set (\textit{cf.} \eqref{gwm})
\[
G_t[m](f)=e^{-t}f+(1-e^{-t})R(f)M+R(f) v\cdot \nabla_x\int_0^t e^{-(t-s)}\mathtt{m}(s;m)ds
\]
and we denote by
\[
\Theta_t(f,m)= \tilde{P}^\sharp_t\varphi(f,m)-\dual{\tilde{\mu}_{R(f)}}{\varphi}=\E\left[\varphi(G_t[m](f),\mathtt{m}(t;m))-\varphi(R(f)\bar{M}_0,\bar{\mathtt{m}}_0)\right]
\]
the integrand in \eqref{correctorphipsi}. For each $t\geq 0$, $\Theta_t$ is differentiable with respect to $f$, the expression of $D_f \Theta_t(f,m)$ being given by
\begin{equation}\label{dTheta}
(g,D_f \Theta_t(f,m))=\E\left[(G_t[m](g),D_f\varphi(G_t[m](f),\mathtt{m}(t;m)))-(R(g)\bar{M}_0,D_f\varphi(R(f)\bar{M}_0,\bar{\mathtt{m}}_0))\right].
\end{equation}
We can write \eqref{dTheta} under the form
\begin{equation}\label{dTheta2}
(g,D_f \Theta_t(f,m))=\tilde{\tilde{P}}^\sharp_t (g,D_f\varphi(f,m))-\dual{\tilde{\tilde{\mu}}_{R(g),R(f)}}{(g,D_f\varphi(f,m))},
\end{equation}
where $\tilde{\tilde{P}}^\sharp_t$, \textit{resp.} $\tilde{\tilde{\mu}}_{R(g),R(f)}$, denote the extended semi-group
\[
\tilde{\tilde{P}}^\sharp_t\varphi(g,f,w,m)=\E\left[\varphi(G_t[m](g),G_t[m](f),\mathtt{w}_t,\mathtt{m}(t;m))\right],
\]
\textit{resp.} invariant measure $\dual{\tilde{\tilde{\mu}}_{\rho_1,\rho_2}}{\varphi}=\E\left[\varphi(\rho_1\bar{M}_0,\rho_2\bar{M}_0,\bar{\mathtt{w}}_0,\bar{\mathtt{m}}_0)\right]$,
with $\mathtt{w}_t$ defined in \eqref{gwm}. The state space of this extended process is then
\[
\tilde{\tilde{\mathcal{X}}}=L^1(\T^d\times V)\times\tilde{\mathcal{X}},\quad\|(g,f,w,m))\|_{\tilde{\tilde{\mathcal{X}}}}=\|g\|_{L^1(\T^d\times V)}+\|(f,w,m)\|_{\tilde{\mathcal{X}}}.
\]
It is simple to extend Proposition~\ref{prop:prePoissonqwm} and Corollary~\ref{cor:prop:prePoissonqwm} to show at once that $\psi$ is differentiable with respect to $f$ and satisfies the properties \eqref{regxvarphi} and \eqref{regxvarphiLip} in restriction to $m=m^\prime$.
We include some details of the proof of the property \eqref{regxvarphi} for $\psi$ for completeness. For some given set of data $X:=(f,w,m)$ and $X^\prime:=(f^\prime,w^\prime,m^\prime)$ in $\tilde{\mathcal{X}}$, we denote by $(X_t)$ and $(X^\prime_t)$ the associated trajectories. As in the proof of Proposition~\ref{prop:prePoissonqwm}, we construct a coupling $(\hat{X}_t,\hat{X}_t^\prime)_{t\geq 0}$ of $(X_t,X_t^\prime)_{t\geq 0}$ which satisfies, $\PP$-almost surely, for every $t\geq 0$, the estimates \eqref{coupling-diff0}. We also set 
\begin{equation}\label{hatgt}
\hat{g}_t=e^{-t}g+(1-e^{-t})R(g)M+R(g) v\cdot \nabla_x\int_0^t e^{-(t-s)}\hat{\mathtt{m}}_sds.
\end{equation}
We fix first $X=(f,w,m)$, and then choose $X^\prime=(f^\prime,w^\prime,m^\prime)$ according to the law $\tilde{\mu}_{R(f)}$. We consider then \eqref{dTheta}, the argument being $\partial_{x_i}g$ instead of $g$, and decompose the difference in the right-hand side of \eqref{dTheta} into the sum of four terms. The first term is
\begin{multline}\label{dThetaDiff1}
\E\left[(\partial_{x_i} [G_t[m](g)],D_f\varphi(G_t[m](f),\mathtt{m}(t;m))-D_f\varphi(R(f)\bar{M}_0,\bar{\mathtt{m}}_0))\right]\\
=\E\int_{\tilde{\mathcal{X}}}\left[(\partial_{x_i} \hat{g}_t,D_f\varphi(\hat{f}_t,\hat{\mathtt{m}}_t)-D_f\varphi(\hat{f}^\prime_t,\hat{\mathtt{m}}_t^\prime))\right]d\tilde{\mu}_{R(f)}(X^\prime).
\end{multline}
In what follows, let us denote by $C$ an arbitrary constant that depends uniquely on the size of $X$ in $\tilde{\mathcal{X}}$ (it does not depend of $g$ in particular). By \eqref{regxvarphiLip}, we can estimate \eqref{dThetaDiff1} from above by
\begin{equation}\label{dThetaDiff1b}
C\E\int_{\tilde{\mathcal{X}}}\left[\|\hat{g}_t\|_{L^1(\T^d\times V)}\|\hat{X}_t-\hat{X}_t^\prime\|_{\tilde{\mathcal{X}}}\right]d\tilde{\mu}_{R(f)}(X^\prime).
\end{equation}
Regarding the dependence in $g$,  we have the bound $\|\hat{g}_t\|_{L^1(\T^d\times V)}
\leq (1+\mathtt{b})\|g\|_{L^1(\T^d\times V)}$, which follows directly from \eqref{hatgt}. We deduce that \eqref{dThetaDiff1} is bounded by $C\|g\|_{L^1(\T^d\times V)}\Gamma_\mathrm{mix}(t)$. The second term that we have to consider involves the commutators of $G_t[m]$ and $\partial_{x_i}$. It reads
\begin{equation}\label{dThetaDiff2}
\E\left[(\{G_t[m],\partial_{x_i}\}g,D_f\varphi(G_t[m](f),\mathtt{m}(t;m))-D_f\varphi(R(f)\bar{M}_0,\bar{\mathtt{m}}_0))\right].
\end{equation}
We have
\[
\{G_t[m],\partial_{x_i}\}g=-gv\cdot\nabla_x\partial_{x_i}\int_0^t e^{-(t-s)}\mathtt{m}(s;m)ds,
\]
hence $\|\{G_t[m],\partial_{x_i}\}g\|_{L^1(\T^d\times V)}\leq \mathtt{b}\|g\|_{L^1(\T^d\times V)}$.
We use the fact that $D_f\varphi$ is Lipschitz continuous on bounded sets and the estimates in \eqref{coupling-diff0} to get the same bound as for  \eqref{dThetaDiff1} on the term \eqref{dThetaDiff2}. The third term in our decomposition is
\begin{equation}\label{dThetaDiff3}
\E\left[(\partial_{x_i} [R(g)\bar{M}_0],D_f\varphi(G_t[m](f),\mathtt{m}(t;m))-D_f\varphi(R(f)\bar{M}_0,\bar{\mathtt{m}}_0))\right].
\end{equation}
It can be estimated (in a slightly simpler way) like \eqref{dThetaDiff1}. The fourth and last term 
\begin{equation}\label{dThetaDiff4}
-\E\left[(R(g)\partial_{x_i} \bar{M}_0,D_f\varphi(G_t[m](f),\mathtt{m}(t;m))-D_f\varphi(R(f)\bar{M}_0,\bar{\mathtt{m}}_0))\right]
\end{equation}
also satisfies the same estimate. Integrating in time now  the inequality $(\partial_{x_i }g,D_f \Theta_t(f,m))\leq C\|g\|_{L^1(\T^d\times V)}\Gamma_\mathrm{mix}(t)$, we obtain \eqref{regxvarphi} for the function $\psi$. At this stage, we can apply the first result of our proposition to $\psi$, to deduce that $\psi\in D(\LL^b)$. To prove that $\psi\in D(\LL^\sharp)$, we fix $(f,m)\in\mathcal{X}$. We want to look at the difference $\tilde{P}^\sharp_t\psi(f,m)-\psi(f,m)$. Since $G_t[m](f)$ has the moment $R(G_t[m](f))$ independent of $t$, we can restrict ourselves to the space $\mathcal{X}_\rho$, where $\rho:=R(f)$. By Proposition~\ref{prop:prePoissonqwm} and Proposition~\ref{prop:Poisson}, we obtain the pointwise limit
\[
\lim_{t\to 0}\frac{1}{t}\left(\tilde{P}^\sharp_t\psi(f,m)-\psi(f,m)\right)=\varphi(f,m).
\]
More precisely, we have (\textit{cf.} \eqref{invPoisson3})
\begin{equation}\label{incrementpsi-generator}
\frac{1}{t}\left(\tilde{P}^\sharp_t\psi(f,m)-\psi(f,m)\right)=\int_0^1\tilde{P}^\sharp_{ts}\varphi(f,m)ds.
\end{equation}
We see on this expression that the increments are additionally bounded on bounded sets. We deduce that
$\psi\in D(\LL^\sharp)$ and that the Poisson equation $\LL^\sharp\psi=\varphi$ is satisfied. 
Once we have shown that $\psi$ is in both $D(\LL^\flat)$ and $D(\LL^\sharp)$, it is expected that 
\begin{equation}\label{psi-in-DLL}
\psi\in D(\LL^\eps),\quad\LL^\eps\psi=\frac{1}{\eps^2}\varphi+\frac{1}{\eps}\LL^\flat\psi. 
\end{equation}
Let us assume that $\eps=1$ for simplicity. We first observe that the commutator between $P^\sharp_t$ and $P^\flat_s$ satisfies
\begin{equation}\label{CommPt}
|\{P^\sharp_t,P^\flat_s\}\psi(f,m)|\leq C(\|(f,m)\|_\mathcal{X},\psi) t(s+\omega(f,s)),
\end{equation}
where $\omega(f,s)$ denote the modulus of continuity
\[
\omega(f,s)=\sup_{|y|\leq s}\iint_{\T^d\times V}|f(x+y,v)-f(x,v)|dxdv,
\]
and where $C(\|(f,m)\|_\mathcal{X},\psi)$ depends on $\|(f,m)\|_\mathcal{X}$ and on the Lipschitz bound of $\psi$ on a bounded set of $\mathcal{X}$ whose size is determined by $\|(f,m)\|_\mathcal{X}$. Indeed, we have 
\begin{align}
|\{P^\sharp_t,P^\flat_s\}\psi(f,m)|
&=\left|\E\left[\psi(G_t[m](f)\circ\Phi_{-s},m)-\psi(G_t[m](f\circ\Phi_{-s}),m)\right]\right|\nonumber\\
&\leq C(\|(f,m)\|_\mathcal{X},\psi)\|G_t[m](f)\circ\Phi_{-s}-G_t[m](f\circ\Phi_{-s})\|_{L^1(\T^d\times V)}.\label{commutatorGPhi}
\end{align}
The commutator $\{G_t[m],\Phi_{-s}\}f:=G_t[m](f)\circ\Phi_{-s}-G_t[m](f\circ\Phi_{-s})$ that appears in \eqref{commutatorGPhi} is
\begin{multline}\label{commutatorGPhi2}
\left[(1-e^{-t})M+v\cdot\nabla_x\int_0^t e^{-(t-r)}\mathtt{m}(r;m)dr\right](R(f\circ\Phi_{-s})-R(f)\circ\Phi_{-s})\\
+R(f)\circ\Phi_{-s}v\cdot\nabla_x\int_0^t e^{-(t-r)}(\mathtt{m}(r;m)-\mathtt{m}(r;m)\circ\Phi_{-s})dr.
\end{multline}
In \eqref{commutatorGPhi2}, the $L^1(\T^d\times V)$-norm of $R(f\circ\Phi_{-s})-R(f)\circ\Phi_{-s}$ is bounded by $\omega(f,s)$ and $\mathtt{m}(r;m)$ is uniformly bounded in a space of functions of class $C^1$. From those two facts, we deduce that the norm of $\{G_t[m],\Phi_{-s}\}f$ is bounded by $C(\|(f,m)\|_\mathcal{X},\psi) t(s+\omega(f,s))$. Taking \eqref{commutatorGPhi} into account, \eqref{CommPt} follows. For similar reasons, we have 
\begin{equation}\label{CommPt2}
|(\mathscr{P}_t-P^\sharp_tP^\flat_t)\psi(f,m)|\leq C(\|(f,m)\|_\mathcal{X},\psi) t(t+\omega(f,t)).
\end{equation}
As a consequence of \eqref{CommPt} and \eqref{CommPt2}, we obtain the decomposition
\begin{equation}\label{Pflatsharp}
\mathscr{P}_t\psi(f,m)-\psi(f,m)=P_t^\flat\circ P_t^\sharp\psi(f,m)-\psi(f,m)+t\eta_t(f,m),
\end{equation}
where the remainder $\eta_t$ is $\pi$-converging to $0$. Once \eqref{Pflatsharp} is established, we use the fact that $\psi\in D(\LL^\flat)\cap D(\LL^\sharp)$ to write successively
\begin{align}
\mathscr{P}_t\psi(f,m)-\psi(f,m)
&=P_t^\flat\circ P_t^\sharp\psi(f,m)-\psi(f,m)+t\eta_t(f,m)\nonumber\\
&=P_t^\flat\left[P_t^\sharp\psi(f,m)-\psi(f,m)\right]+P_t^\flat\psi(f,m)-\psi(f,m)+t\eta_t(f,m)\label{decPflatsharp}.
\end{align}
The three last terms in \eqref{decPflatsharp} give a contribution $t(\LL^\flat\psi(f,m)+\eta^\prime_t(f,m))$, where the remainder $\eta^\prime_t$ is $\pi$-converging to $0$. The first term in \eqref{decPflatsharp} is
\begin{equation}\label{decPflatsharp1}
t P_t^\flat\int_0^1 P_{st}^\sharp\LL^\sharp\psi(f,m)ds=t P_t^\flat\int_0^1 P_{st}^\sharp\varphi(f,m)ds=t \int_0^1 P_t^\flat P_{st}^\sharp\varphi(f,m)ds.
\end{equation}
By the local Lipschitz property of $\varphi$, \eqref{decPflatsharp1}  has the asymptotic expansion $t\varphi(f,m)+t\eta_t''(f,m)$, where the remainder $\eta_t''$ is $\pi$-converging to $0$. We can therefore conclude, finally, to \eqref{psi-in-DLL}.
\end{proof}

We complete Proposition~\ref{prop:admtest1} with a computation of the function $\psi$ given by \eqref{correctorphipsi} in two special cases (Lemma~\ref{lemma:admtest1} and Lemma~\ref{lemma:admtest2}). We will use the following notations: for $\rho\in L^1(\T^d)$ and $\varphi$ defined on $\mathcal{X}$, $H_\rho\varphi$ is the function defined on $\mathtt{F}\times\mathtt{F}$ by 
\begin{equation}\label{Hrho}
H_\rho\varphi(n,m)=\varphi(\rho v\cdot\nabla_x n,m).
\end{equation}
If $\theta\colon\mathtt{F}\times\mathtt{F}\to\R$, we denote by $\pi_\Delta$ the map $m\mapsto(m,m)$ and we set
\begin{equation}\label{piDelta}
{\pi_\Delta}\theta(m)=\theta\circ\pi_\Delta(m)=\theta(m,m).
\end{equation}
We also denote by $P_t^{[2]}\theta(n,m)$ and $U_\alpha^{[2]}\theta(n,m)$ the quantities
\begin{equation}\label{[2]theta}
P_t^{[2]}\theta(n,m)=\E\theta(n,\mathtt{m}(t;m)),\quad U_\alpha^{[2]}\theta(n,m)=\int_0^\infty e^{-\alpha t}P_t^{[2]}\theta(n,m) dt,
\end{equation}
whenever they make sense. 

\begin{lemma} In the context of Proposition~\ref{prop:admtest1}, we suppose that $\varphi$ is a good test function. We assume furthermore that, for a given $\rho\in L^1(\T^d)$, and in restriction to $\mathcal{X}_\rho$, the function $\varphi$ is bilinear in the argument $(f,m)$. We have then, using the notations \eqref{Hrho}-\eqref{piDelta}-\eqref{[2]theta},
\begin{equation}\label{bilinear-equilibrium}
\dual{\varphi}{\tilde{\mu}_\rho}=\dual{\pi_\Delta U_1^{[2]} H_\rho \varphi }{\lambda},
\end{equation}
and, when the quantity in \eqref{bilinear-equilibrium} vanishes,
\begin{equation}\label{bilinear-inverse}
-\psi(f,m)=U_1\varphi(f-\rho M,m)+U_0\varphi(\rho M,m)+\rho U_0(\pi_\Delta U_1^{[2]} H_\rho \varphi)(m),
\end{equation}
for all $(f,m)\in\mathcal{X}_\rho$.
\label{lemma:admtest1}\end{lemma}

\begin{proof}[Proof of Lemma~\ref{lemma:admtest1}] Using the bilinear character of $\varphi$, we can write $\tilde{P}^\sharp_t\varphi(f,m)$ for $(f,m)\in\mathcal{X}_\rho$, as the sum of the two terms
\begin{equation}\label{lemadm1}
\E\left[\varphi(f-\rho M,e^{-t}\mathtt{m}(t;m))+\varphi(\rho M,\mathtt{m}(t;m))\right],
\end{equation}
and
\begin{equation}\label{lemadm2}
\E\left[\varphi\left(\rho v\cdot\nabla_x\int_0^t e^{-(t-s)}\mathtt{m}(s;m),\mathtt{m}(t;m)\right)\right].
\end{equation}
The first term \eqref{lemadm1} is 
\begin{equation}\label{lemadm12}
\eqref{lemadm1}=e^{-t}P_t\varphi(f-\rho M,m)+P_t\varphi(\rho M,m),
\end{equation}
where $(P_t)$ is the semi-group associated to $(\mathtt{m}_t)$ (recall that the semi-group associated to $(f_t,\mathtt{m}_t)$ is denoted by $(\mathscr{P}_t)$). By bilinearity and the Markov formula, the second term \eqref{lemadm2} is 
\begin{equation}\label{lemadm22}
\eqref{lemadm2}=\int_0^t e^{-(t-s)}\E\left[(\pi_\Delta P_{t-s}^{[2]}H_\rho\varphi)(\mathtt{m}(s;m))\right]ds,
\end{equation}
which gives
\begin{equation}\label{lemadm23}
\eqref{lemadm2}=\int_0^t e^{-(t-s)}P_s(\pi_\Delta P_{t-s}^{[2]}H_\rho\varphi)(m)ds=\int_0^t e^{-s}P_{t-s}(\pi_\Delta P_s^{[2]}H_\rho\varphi)(m)ds.
\end{equation}
By passing to the limit $t\to+\infty$ in \eqref{lemadm12} and \eqref{lemadm23} (we use the dominated convergence theorem for \eqref{lemadm23}), we obtain 
\begin{equation}\label{bilinear-equilibrium0}
\dual{\varphi}{\tilde{\mu}_\rho}=\int_\mathtt{F}\left[\varphi(\rho M,m)+(\pi_\Delta U_1^{[2]} H_\rho \varphi)(m)\right]d\lambda(m).
\end{equation}
The first term in the right hand-side of \eqref{bilinear-equilibrium0} is $\varphi\left(\rho M,\E\left[\bar{\mathtt{m}}_0\right]\right)$ (here we use the bilinear and continuous characters of $\varphi$). Since $\E\left[\bar{\mathtt{m}}_0\right]=0$ by hypothesis, \eqref{bilinear-equilibrium0} gives \eqref{bilinear-equilibrium}. Let us assume now that $\dual{\varphi}{\tilde{\mu}_\rho}=0$, which means that 
\begin{equation}\label{cancelP[2]}
\dual{\pi_\Delta P_s^{[2]} H_\rho\varphi}{\lambda}=0.
\end{equation} 
By integration with respect to $t\in\R_+$ in \eqref{lemadm12} and \eqref{lemadm23}, we get
 \begin{equation}\label{bilinear-inverse0}
-\psi(f,m)=U_1\varphi(f-\rho M,m)+U_0\varphi(\rho M,m)+\rho\int_0^\infty\int_0^t e^{-s}P_{t-s}(\pi_\Delta P_s^{[2]} H_\rho \varphi)(m)ds dt.
\end{equation}
The last integral in \eqref{bilinear-inverse0} is absolutely convergent (this is a consequence of \eqref{mixCoupled} and \eqref{cancelP[2]}). We can use the Fubini theorem to get
\[
\int_0^\infty\int_0^t e^{-s}P_{t-s}(\pi_\Delta P_s^{[2]}H_\rho \varphi)(m)ds dt=\int_0^\infty\int_t^\infty e^{-s}P_{t-s}(\pi_\Delta P_s^{[2]}H_\rho \varphi)(m)dt ds.
\]
The change of variable $t^\prime=t-s$ shows then that it is equal to the last term $U_0(\pi_\Delta U_1^{[2]} H_\rho\varphi)(m)$ in \eqref{bilinear-inverse}.
\end{proof}

\begin{remark}\label{rk:linInf} The computations used in the proof of Lemma~\ref{lemma:admtest1} show that if, in restriction to $\mathcal{X}_\rho$, the function $\varphi$ depends on $f$ only, and in a linear way, then $\dual{\varphi}{\tilde{\mu}_\rho}=\varphi(\rho M)$ and the function $\psi$ defined in \eqref{correctorphipsi} has the expression
\begin{equation}\label{linear-inverse}
-\psi(f,m)=\varphi(f-\rho M+v\cdot\nabla_x U_0 m)
\end{equation}
for all $(f,m)\in\mathcal{X}_\rho$, where $U_0 m$ stands for $U_0\mathrm{Id}(m)$, when $\mathrm{Id}$ is the identity of $F$.
\end{remark}

In the next proposition we examine the case where, as in Remark~\ref{rk:linInf}, $\varphi$ in restriction to $\mathcal{X}_\rho$ depends on $f$ only, but in a bilinear way.

\begin{lemma} In the context of Proposition~\ref{prop:admtest1}, we suppose that $\varphi$ is a good test function. We assume furthermore that, for a given $\rho\in L^1(\T^d)$, and in restriction to $\mathcal{X}_\rho$, the function $\varphi$ is independent of $m$ and bilinear in the argument $f$. We have then
\begin{equation}\label{bilinear-equilibrium-2}
\dual{\varphi}{\tilde{\mu}_\rho}=\varphi(\rho M)+\dual{\pi_\Delta U_1^{[2]}H_\rho \varphi}{\lambda}.
\end{equation}
\label{lemma:admtest2}\end{lemma}

\begin{proof}[Proof of Lemma~\ref{lemma:admtest2}] Let us write $\varphi(f,f)$ for $\varphi(f)$ to indicate the bilinear dependence in $f$. This gives
\begin{equation}\label{bilinear-equilibrium-20}
\dual{\varphi}{\tilde{\mu}_\rho}=\E\left[\varphi(\rho M+\rho v\cdot\nabla_x\bar{\mathtt{w}}_0,\rho M+\rho v\cdot\nabla_x\bar{\mathtt{w}}_0)\right]
=\varphi(\rho M)+\E\left[\varphi(\rho v\cdot\nabla_x\bar{\mathtt{w}}_0)\right],
\end{equation}
since $\E\left[\bar{\mathtt{w}}_0\right]=0$ by \eqref{mcentred}. The last term in \eqref{bilinear-equilibrium-20} can be developed as
\begin{equation}\label{bilinear-equilibrium-21}
\E\left[\varphi(\rho v\cdot\nabla_x\bar{\mathtt{w}}_0)\right]=\rho^2\int_{-\infty}^0\int_{-\infty}^0 e^{t+s}\E\varphi(v\cdot\nabla_x\bar{\mathtt{m}}_t,v\cdot\nabla_x\bar{\mathtt{m}}_s)dsdt.
\end{equation}
By symmetry and the Markov property, the quantity in \eqref{bilinear-equilibrium-21} is
\begin{align}
2\rho^2\int_{-\infty}^0\int_{t}^0 e^{t+s}\E \pi_\Delta P^{[2]}_{s-t}H_\rho\varphi(\bar{\mathtt{m}}_s)dtds
&=2\rho^2\int_{-\infty}^0\int_{t}^0 e^{t+s}\E \pi_\Delta P^{[2]}_{s-t}H_\rho\varphi(\bar{\mathtt{m}}_s)dtds\nonumber\\
&=2\rho^2\int_{-\infty}^0\int_{t}^0 e^{t+s}\dual{\pi_\Delta P^{[2]}_{s-t}H_\rho\varphi}{\lambda}dtds.
\label{bilinear-equilibrium-22}\end{align}
By Fubini's theorem and change of variable, we obtain the following expressions
\begin{align}
2\rho^2\int_{-\infty}^0\int_{0}^{-t} e^{2t+s}\dual{\pi_\Delta P^{[2]}_{s}H_\rho\varphi}{\lambda}dtds&=2\rho^2\int_0^\infty \int_{-\infty}^{-s} e^{2t+s}\dual{\pi_\Delta P^{[2]}_{s}H_\rho\varphi}{\lambda}ds dt\nonumber\\
&=\rho^2\int_0^\infty e^{-s} \dual{\pi_\Delta P^{[2]}_{s}H_\rho\varphi}{\lambda}ds,\label{bilinear-equilibrium-23}
\end{align}
which yield \eqref{bilinear-equilibrium-2}.
\end{proof}

\begin{remark}\label{rk:square} Examination of \eqref{bilinear-equilibrium-21}-\eqref{bilinear-equilibrium-22}-\eqref{bilinear-equilibrium-23}, shows that, for all $u\in L^2(\T^d;\R^d)$ and $\alpha>0$, we have the formula
\begin{equation}\label{eq:square}
\alpha\E\left[\left|\int_{-\infty}^0 e^{\alpha t}\dual{u}{\chi(\bar{\mathtt{m}}_t)}dt\right|^2\right]=\E\left[\iint_{\T^d\times\T^d}(U_\alpha\chi_i)(\bar{\mathtt{m}}_0)(x)\chi_j(\bar{m_0})(y)u(x)u(y) dx dt\right].
\end{equation} 
\end{remark}

\section{The perturbed test-function method}\label{sec:perturbed}
In the context of Theorem~\ref{th:mainth}, our first objective is to prove the convergence in law of the sequence $(\rho^{\eps}_t)_{\eps}$. To that purpose, we use the perturbed test function method devised in \cite{PapanicolaouStroockVaradhan77}. We start by considering a function $\varphi$ of the variable $\rho\in L^1(\T^d)$. We lift it to the function defined on the state space $\mathcal{X}$, given by $(f,m)\mapsto\varphi(R(f))$. By abuse of notations, we still denote by $\varphi$ this function. The principle of the perturbed test function method is to find two functions $\varphi_1$, $\varphi_2$ in the domain of $\LL^\eps$ such that, for the modified test-function
\begin{equation}\label{perturrbed}
\varphi^\eps:=\varphi+\eps\varphi_1+\eps^2\varphi_2,
\end{equation}
we have the following asymptotic expansion
\begin{equation}\label{controlledLeps}
\LL^\eps\varphi^\eps=\LL\varphi+o(1).
\end{equation}

This gives the limit generator $\LL$, as we will see in Section~\ref{sec:corrector2}. Of course we must choose $\varphi$ to ensure that $\varphi^\eps$ given by \eqref{perturrbed} is in the domain of $\LL^\eps$ (see Proposition~\ref{prop:varphiADM} below).
%
%

\subsection{First and second correctors}\label{sec:correctors}
We insert the development \eqref{perturrbed} in \eqref{controlledLeps} and proceed by identification of the different powers of $\eps$ to get the following equations:
\begin{align}
\LL^\sharp\varphi&=0,\label{varphi-order-2}\\
\LL^\sharp\varphi_1+\LL^\flat\varphi&=0,\label{varphi1-order-1}\\
\LL^\sharp\varphi_2+\LL^\flat\varphi_1&=\LL\varphi.\label{varphi2-order-0}
\end{align}

Before we analyse the equation~\eqref{varphi1-order-1}, we first justify that the test function $\varphi$ is admissible and satisfies \eqref{varphi-order-2}. 

\begin{proposition} Let $\xi\in C^\infty(\T^d)$ and $\psi\colon\R\to\R$ be of class $C^3$. Then the function
\begin{equation}\label{varphiPerturbed}
\varphi\colon(f,m)\to\psi\Big(\dual{R(f)}{\xi}\Big)
\end{equation}
is in the intersection of the three domains $D(\LL^\sharp)$, $D(\LL^\flat)$, $D(\LL^\eps)$ and satisfies \eqref{varphi-order-2}. By the action of $\LL^\flat$, it yields a function
\begin{equation}\label{Lflatvarphi}
\LL^\flat\varphi\colon(f,m)\to-\psi'(\dual{R(f)}{\xi})\dual{J(f)}{\nabla_x\xi}
\end{equation}
which is a good test function.
\label{prop:varphiADM}\end{proposition}

\begin{proof}[Proof of Proposition~\ref{prop:varphiADM}] The function $\varphi$ is constant along the flow of $\LL^\sharp$: it is trivially in $D(\LL^\sharp)$, with $\LL^\sharp\varphi=0$. That $\varphi\in D(\LL^\flat)$ and Formula~\eqref{Lflatvarphi} for $\LL^\flat\varphi$ result from the duality
\[
\dual{R(f\circ\Phi_{-t})}{\xi}=\iint_{\T^d\times V}f(x-tv,v)\xi(x)dxd\nu(v)=\iint_{\T^d\times V}f(x,v)\xi(x+tv)dxd\nu(v).
\]
Clearly, $\LL^\flat\varphi$ is Lipschitz continuous on bounded sets and differentiable with respect to $f$, with the following expression:
\begin{equation}\label{DfLflatvarphi}
(g,D_f\LL^\flat\varphi(f,m))=\psi'\Big(\dual{R(f)}{\xi}\Big)\dual{J(g)}{\nabla_x\xi}+\psi''\Big(\dual{R(f)}{\xi}\Big)\dual{J(f)}{\nabla_x\xi}\dual{R(g)}{\xi}.
\end{equation} 
It is then obvious that $\LL^\flat\varphi$ is a good test function.
\end{proof}

\subsubsection{First corrector}\label{sec:corrector1}

Let $\varphi$ be fixed, and as in \eqref{varphiPerturbed}. We apply Proposition~\ref{prop:admtest1} and Proposition~\ref{prop:varphiADM}: this provides a solution $\varphi_1$ to the second equation \eqref{varphi1-order-1} defined by
\begin{equation}\label{varphi1-1}
\varphi_1(f,m)=\int_0^\infty\left[ \tilde{P}^\sharp_t\LL^\flat\varphi(f,m)-\dual{\tilde{\mu}_{R(f)}}{\LL^\flat\varphi}\right]dt.
\end{equation}
Let $\rho\in L^1(\T^d\times F)$. In restriction to $\mathcal{X}_\rho$, we have $\LL^\flat\varphi(f,m)=\psi'\Big(\dual{\rho}{\xi}\Big)\dual{J(f)}{\nabla_x\xi}$, which is a bilinear function of $(f,m)$. Using the Lemma~\ref{lemma:admtest1} to give an explicit expression for $\varphi_1$, we have to compute $J(v\cdot\nabla_x m)$, which is $\chi(m)$ by Definition~\eqref{defchin}. We obtain the formula
\begin{equation}\label{Lflatvarphimm}
\LL^\flat\varphi(\rho v\cdot\nabla_x m,m)=\psi'\Big(\dual{\rho}{\xi}\Big)\dual{\rho\chi(m)}{\nabla_x\xi}.
\end{equation}
It follows from \eqref{bilinear-equilibrium} and the cancellation property \eqref{mcentred} that $\dual{\LL^\flat\varphi}{\tilde{\mu}_\rho}=0$. Since $\LL^\flat\varphi(f,m)$ is actually independent of $m$, simple computations based on \eqref{bilinear-inverse} give the expression
\begin{equation}\label{varphi1-2}
\varphi_1(f,m)=\psi'\Big(\dual{R(f)}{\xi}\Big)\dual{J(f)+R(f)U_0\chi(m)}{\nabla_x\xi}.
\end{equation}

\begin{proposition}[First corrector] The first corrector $\varphi_1\colon\mathcal{X}\to\R$ given by \eqref{varphi1-2} has the following properties:
\begin{enumerate}
\item\label{item-varphi1-1} the function $\varphi_1$ is a good test function,
\item\label{item-varphi1-2} the function $\varphi_1$ is in the intersection of the three domains $D(\LL^\sharp)$, $D(\LL^\flat)$, $D(\LL^\eps)$ and \eqref{varphi1-order-1} is satisfied,
\item\label{item-varphi1-3} the function $\LL^\flat\varphi_1$ is a good test function.
\end{enumerate} 
\label{prop:first-corrector}\end{proposition}

\begin{proof}[Proof of Proposition~\ref{prop:first-corrector}] We already know by the identity \eqref{varphi1-1} and Proposition~\ref{prop:admtest1} that $\varphi_1$ is an admissible test function, that $\varphi_1$ is in the intersection of the three domains $D(\LL^\sharp)$, $D(\LL^\flat)$, $D(\LL^\eps)$ and that \eqref{varphi1-order-1} is satisfied. The expression \eqref{varphi1-2} shows that $\varphi_1$ is a good test function and that $\LL^\flat\varphi_1$ is given by
\begin{align}
\LL_\flat\varphi_1(f,m)&=\psi''\Big(\dual{R(f)}{\xi}\Big)\dual{J(f)}{\nabla_x\xi}\dual{J(f)+ R(f)U_0\chi(m)}{\nabla_x\xi}\label{limGen1-1}\\
&+\psi'\Big(\dual{ R(f)}{\xi}\Big)\left[\dual{K(f)+J(f)\otimes U_0\chi(m)}{D^2_x\xi}+\dual{J(f)}{\nabla_x U_0\chi(m)\nabla_x\xi}\right],\label{limGen1-2}
\end{align}
where $\nabla_x U_0\chi(n)$ is the matrix with $ij$-entry $\partial_{x_i}U_0\chi^j(n)$ (we also refer to Section~\ref{sec:notations} for the meaning of $J(f)\otimes U_0\chi(m)$ in \eqref{limGen1-2}). It is clear, on the basis of \eqref{limGen1-1}-\eqref{limGen1-2}, that $\LL^\flat\varphi_1$ is a good test function.
\end{proof}

\begin{proposition}[Perturbation of order one] Let $\varphi_1\colon\mathcal{X}\to\R$ be given by \eqref{varphi1-2}. Then the square $|\varphi+\eps\varphi_1|^2$ is in the domain of $\LL^\eps$ and
\begin{equation}\label{LepsSquareOrderone}
\LL^\eps\left[|\varphi+\eps\varphi_1|^2\right]=\mathtt{A}|\varphi_1|^2-2\varphi_1\mathtt{A}\varphi_1.
\end{equation}
In particular, $\LL^\eps\left[|\varphi+\eps\varphi_1|^2\right]$ is bounded on bounded sets independently of $\eps$: for all $r_0>0$, there exists $r_1>0$ such that $|\LL^\eps\left[|\varphi+\eps\varphi_1|^2\right](f,m)|\leq r_1$ for all $\eps\in[0,1]$ and for all $(f,m)$ in the ball $\bar{B}(0,r_0)$ of $\mathcal{X}$.
\label{prop:SquareOrderone}\end{proposition}

\begin{proof}[Proof of Proposition~\ref{prop:SquareOrderone}] To establish the proposition, we need to see the generator $\LL^\eps$ not as the weighted sum of $\LL^\sharp$ and $\LL^\flat$, as we have done up to now, but as the sum $\eps^{-2}\mathtt{A}+\mathscr{K}^\eps$, where
\begin{equation}\label{defKeps}
\mathscr{K}^\eps\varphi(f,m)=(\eps^{-2}(L(f)+R(f)v\cdot\nabla_x m)-\eps^{-1}v\cdot\nabla_x f,D_f\varphi(f,m)).
\end{equation}
Let us specify this aspect in the case $\eps=1$ for simplicity. Given $m\in\mathtt{F}$, we consider the semi-group $(\mathscr{Q}_t[m])$ acting on functions $\varphi(f)$ by the formula $\mathscr{Q}_t[m]\varphi=\varphi\circ\Upsilon_t[m]$, where $(\Upsilon_t[m])$ is the flow associated to the resolution of the PDE
\begin{equation}\label{PDEQ}
\partial_t f_t+v\cdot\nabla_x f_t=R(f_t)(M+v\cdot\nabla_x m)-f_t.
\end{equation}
If $f_0=f$, then $\Upsilon_t[m](f)$ solves the fixed-point equation
\begin{equation}\label{solvePDEQ}
\Upsilon_t[m](f)=e^{-t}f\circ\Phi_{-t}+\int_0^t e^{-(t-s)}[ R(\Upsilon_s[m](f)) (M+v\cdot m)]\circ\Phi_{-(t-s)}ds.
\end{equation}
The extension to functions $\varphi(f,m)$ as in \eqref{defKeps} is done trivially by setting $\mathscr{Q}_t[m]\varphi(f,m)=\varphi(\Upsilon_t[m](f),m)$: all the components $m$ are frozen here. We claim that, if $\phi$ is an admissible test function, then $|\phi|^2$ is in the domain of $\mathscr{K}^\eps$ and
\begin{equation}\label{Kphisquare}
\mathscr{K}^\eps\left[|\phi|^2\right]=2\phi\mathscr{K}^\eps\phi.
\end{equation}
We will skip the proof of this assertion. It results from tedious but elementary estimates. At any rate, we deduce, using Proposition~\ref{prop:first-corrector}, that $|\varphi+\eps\varphi_1|^2$ is in the domain of $\mathscr{K}^\eps$. That $|\varphi+\eps\varphi_1|^2$ is in the domain of $\mathtt{A}$ results from Hypothesis~\eqref{AR0}. To conclude next, we use the fact that, for all $m\in\mathtt{F}$, 
\begin{equation}\label{supStoContm}
\lim_{t\to 0}\sup_{0\leq s\leq t}\E\|\mathtt{m}(t;m)-\mathtt{m}(s;m)\|_F=0
\end{equation}
and that, almost surely, 
\begin{equation}\label{compPsif}
\|\Upsilon_t[\mathtt{m}(t;m)](f)-f_t\|_{L^1(\T^d\times V)}\leq\int_0^t \|\mathtt{m}(t;m)-\mathtt{m}(s;m)\|_F ds,
\end{equation}
where $f_t$ in \eqref{compPsif} is the solution to \eqref{mildLB}. The estimate \eqref{compPsif} is analogous to \eqref{sol:fLB12}. To establish \eqref{supStoContm}, we observe that the limit exists by monotony. Assume that it is strictly positive: we can find $\delta>0$ and some sequences $(s_n)\downarrow 0$, $(t_n)\downarrow0$ such that $\E\|\mathtt{m}(t_n;m)-\mathtt{m}(s_n;m)\|_F>\delta$. At the same time, $\E\|\mathtt{m}(t_n;m)-m\|_F$ tends to $0$ (similarly for $\E\|\mathtt{m}(s_n;m)-m\|_F$) since $(\mathtt{m}_t)$ is stochastically continuous and bounded. We obtain a contradiction. Next, we observe that the commutator between $P_t$ (semi-group associated to $(\mathtt{m}_t)$) and $\mathscr{Q}_t$ is given by 
\begin{equation}\label{commutatorPQQ}
\{P_t,\mathscr{Q}_t\}\phi(f,m)=\phi\left(\Upsilon_t[\mathtt{m}(t;m)](f),\mathtt{m}(t;m)\right)-\phi\left(\Upsilon_t[m](f),\mathtt{m}(t;m)\right).
\end{equation}
When $\phi$ is Lipschitz continuous on bounded sets, \eqref{supStoContm} and \eqref{compPsif} give a bound on \eqref{commutatorPQQ} which, up to a multiplicative constant that depends on $\|(f,m)\|_\mathcal{X}$ and on $\phi$, is $t\cdot\eta(t)$, where $\eta(t):=\E\|\mathtt{m}(t;m)-m\|_F$ is a quantity which tends to $0$ with $t$. We have also a similar estimate for $(\mathscr{P}_t-P_t\mathscr{Q}_t)\phi(f,m)$. We can now
follow the line of reasoning of the proof of Item~\ref{item-psiLL} of Proposition~\ref{prop:admtest1}. We conclude to the fact that a function $\phi$ which is Lipschitz continuous on bounded sets and in the intersection of the domains of $\mathscr{K}^\eps$ and $\mathtt{A}$ is in the domain of $\LL^\eps$. Finally, we use \eqref{Kphisquare}, which reads as
\begin{equation}\label{square-1}
(\LL^\eps-\eps^{-2}\mathtt{A})\left[|\varphi+\eps\varphi_1|^2\right]=2(\varphi+\eps\varphi_1)(\LL^\eps-\eps^{-2}\mathtt{A})(\varphi+\eps\varphi_1).
\end{equation}
By expansion of both sides of \eqref{square-1} and by \eqref{varphi-order-2}, \eqref{varphi1-order-1}, we obtain \eqref{LepsSquareOrderone}. Then, we use the explicit form \eqref{varphi1-2} and the identity $\mathtt{A}U_0=-\mathrm{Id}$ to deduce from \eqref{LepsSquareOrderone} the expression
\begin{equation}\label{square-3}
\left|\psi'\Big(\dual{R(f)}{\xi}\Big)\right|^2\left[\mathtt{A}\left|\dual{R(f)U_0\chi(m)}{\nabla_x\xi}\right|^2+2\dual{R(f)U_0\chi(m)}{\nabla_x\xi}\dual{R(f)\chi(m)}{\nabla_x\xi}\right]
\end{equation}
for $\LL^\eps\left[|\varphi+\eps\varphi_1|^2\right](f,m)$. The square term in \eqref{square-3} is of the form $\mathtt{A}\left[\left|\Lambda_f\circ U_0\circ\chi\right|^2\right](m)$,
where $\Lambda_f$ is the linear operator $m\mapsto\dual{R(f) m}{\nabla_x\xi}$. Once this remark is done, we can estimate $\LL^\eps\left[|\varphi+\eps\varphi_1|^2\right]$ by \eqref{square-3} and Hypothesis~\eqref{AR0} to conclude that it is bounded on bounded sets independently of $\eps$.
\end{proof}

\subsubsection{Second corrector and limit generator}\label{sec:corrector2}

We examine now the equation \eqref{varphi2-order-0}, \textit{i.e.} $\LL_\sharp\varphi_2+\LL_\flat\varphi_1=\LL\varphi$. We set
\begin{equation}\label{limGen}
\LL\varphi(\rho)=\dual{\LL_\flat\varphi_1}{\mu_\rho}.
\end{equation}
Equation~\eqref{limGen} defines the limit generator $\LL$. We apply then Proposition~\ref{prop:admtest1} and Proposition~\ref{prop:varphiADM}, using the properties of $\LL^\flat\varphi_1$ described in Proposition~\ref{prop:first-corrector}. We obtain a solution $\varphi_2$ to the equation \eqref{varphi2-order-0} given by
\begin{equation}\label{varphi2-1}
\varphi_2(f,m)=\int_0^\infty\left[ \tilde{P}^\sharp_t\LL^\flat\varphi_1(f,m)-\dual{\tilde{\mu}_{R(f)}}{\LL^\flat\varphi_1}\right]dt.
\end{equation}
We use Lemmas~\ref{lemma:admtest1} and \ref{lemma:admtest2} to get a more explicit expression for $\LL\varphi(\rho)$. The first term in \eqref{limGen1-1} is bilinear in $f$ and independent of $m$. Using Lemma~\ref{lemma:admtest2} and the fact that $J(M)=0$ and $J(v\cdot\nabla_x m)=\chi(m)$ gives the first contribution
\begin{equation}\label{LimGen1-3}
\psi''\big(\dual{\rho}{\xi}\big)\int_\mathtt{F}\dual{\rho\chi(m)}{\nabla_x\xi}\dual{\rho U_1\chi(m)}{\nabla_x\xi}d\lambda(m)
\end{equation}
to $\LL\varphi(\rho)$. The second term in \eqref{limGen1-1} gives, by Lemma~\ref{lemma:admtest1}, the contribution
\begin{equation}\label{LimGen1-4}
\psi''\big(\dual{\rho}{\xi}\big)\int_\mathtt{F}\dual{\rho\chi(m)}{\nabla_x\xi}\dual{\rho U_1U_0\chi(m)}{\nabla_x\xi}d\lambda(m).
\end{equation}
We use the resolvent formula $U_1U_0=U_0-U_1$ to obtain
\begin{equation}\label{LimGen1-5}
\eqref{LimGen1-3}+\eqref{LimGen1-4}=\psi''\big(\dual{\rho}{\xi}\big)\int_\mathtt{F}\dual{\rho\chi(m)}{\nabla_x\xi}\dual{\rho U_0\chi(m)}{\nabla_x\xi}d\lambda(m).
\end{equation}
The contribution to $\LL\varphi(\rho)$ of the first term in \eqref{limGen1-2} is given by Remark~\ref{rk:linInf}, this is
\begin{equation}\label{LimGen1-6}
\psi'\big(\dual{\rho}{\xi}\big)\dual{\rho K(M)}{D^2_x\xi}.
\end{equation}
Using Lemma~\ref{lemma:admtest1}, we compute the contribution of the two last terms in \eqref{limGen1-2}, which is
\begin{equation}\label{LimGen1-7}
\psi'\big(\dual{\rho}{\xi}\big)\left[\dual{\rho\chi(m)\otimes U_1U_0\chi(m)}{D^2_x\xi}
+\dual{\rho\chi(m)}{\nabla_x U_1U_0\chi(m)\nabla_x\xi}\right].
\end{equation}
Adding \eqref{LimGen1-5}, \eqref{LimGen1-6}, \eqref{LimGen1-7} gives the final expression
\begin{multline}\label{limGen2}
\LL\varphi(\rho)=\psi''\big(\dual{\rho}{\xi}\big)\E\left[\dual{\rho \chi(\bar{\mathtt{m}}_0)}{\nabla_x\xi}\dual{\rho U_0\chi(\bar{\mathtt{m}}_0)}{\nabla_x\xi}\right]\\
+\psi'\big(\dual{\rho}{\xi}\big)\E\Big[\dual{\rho K(M)+\rho \chi(\bar{\mathtt{m}}_0)\otimes U_0U_1\chi(\bar{\mathtt{m}}_0)}{D^2_x\xi}\\
+\dual{\rho\chi(\bar{\mathtt{m}}_0)}{\nabla_x U_0U_1\chi(\bar{\mathtt{m}}_0)\nabla_x\xi}\Big]
\end{multline}
of the limit generator $\LL$. A more readable expression of \eqref{limGen2} is actually the following one
\begin{multline}\label{limGen3}
\LL\varphi(\rho)=\Big(K(M) : D^2_x\rho+\E\divv_x\big[\divv_x\big(\rho\chi(\bar{\mathtt{m}}_0)\big) U_0U_1\chi(\bar{\mathtt{m}}_0)\big],D_\rho\varphi(\rho)\Big)\\
+\E D^2_\rho\varphi(\rho)\cdot\big(\divv_x[\rho U_0\chi(\bar{\mathtt{m}}_0)],\divv_x[\rho \chi(\bar{\mathtt{m}}_0)]\big),
\end{multline}
although it is \eqref{limGen2} that makes sense, as the argument $\rho$ is in $L^1(\T^d)$. The relation between \eqref{limGen3} and the limit equation \eqref{eq:limeq} will be established in the next Section~\ref{sec:arrange} and then in Section~\ref{sec:cvDA}. First, we state the following result, which is an immediate consequence of Proposition~\ref{prop:admtest1} and Proposition~\ref{prop:first-corrector}.

\begin{proposition}[Second corrector] Let $\varphi_2\colon\mathcal{X}\to\R$ be defined by \eqref{varphi2-1}. Then $\varphi_2$ is an admissible test function. It is in the intersection of the three domains $D(\LL^\sharp)$, $D(\LL^\flat)$, $D(\LL^\eps)$, and the equation \eqref{varphi2-order-0} is satisfied.
\label{prop:second-corrector}\end{proposition}

\subsection{Analysis of the limit generator}\label{sec:arrange}

Let us give some elements of analysis of $\LL\varphi(\rho)$. In the first order terms of \eqref{limGen3}, we recognize the drift part of \eqref{eq:limeq}, with $K^*$ and $\Psi$ defined by \eqref{Kflat} and \eqref{coeffPsi}. The second order term of \eqref{limGen3} involves the operator $S$ on $L^2(\T^d;\R^d)$ with kernel $\mathtt{C}$ defined by
\begin{equation}\label{SKernelK}
(S u)_i(x)=\sum_{j=1}^d\int_{\T^d}\mathtt{C}_{ij}(x,y) u_j(y)dy,\quad\mathtt{C}_{ij}(x,y)=\E\left[(U_0\chi_i)(\bar{\mathtt{m}}_0)(x)\chi_j(\bar{\mathtt{m}}_0)(y)\right].
\end{equation}
Indeed, if $\varphi(\rho)=\frac12|\dual{\rho}{\xi}|^2$, then the second-order term in \eqref{limGen2} is
\begin{equation}\label{SecondQuad}
\E D^2_\rho\varphi(\rho)\cdot(\divv_x[\rho (U_0\chi)(\bar{\mathtt{m}}_0)],\divv_x[\rho \chi(\bar{\mathtt{m}}_0)])
=\dual{S(\rho\nabla_x\xi)}{\rho\nabla_x\xi}_{L^2(\T^d;\R^d)},
\end{equation}
where the scalar product of $u,v\in L^2(\T^d;\R^d)$ is defined as the integral over $\T^d$ of $u\cdot v$. We will prove the following result.
\begin{proposition} Let $(\bar{\mathtt{m}}_t)$ be an admissible pilot process in the sense of Definition~\ref{def:admbarm}. Then the operator $S$ defined by \eqref{SKernelK} is bounded, symmetric and non-ne\-ga\-tive on $L^2(\T^d;\R^d)$.
\label{prop:sympos}\end{proposition}

\begin{proof}[Proof of Proposition~\ref{prop:sympos}] By \eqref{BallR} and \eqref{boundRchi}, we have
\begin{equation}\label{BoundSC}
|\mathtt{C}_{ij}(x,y)|\leq\mathtt{b}^2\|\gamma_\mathrm{mix}\|_{L^1(\R_+)},\quad \|S\|_{L^2(\T^d;\R^d)\to L^2(\T^d;\R^d)}\leq \mathtt{b}^2\|\gamma_\mathrm{mix}\|_{L^1(\R_+)}.
\end{equation}
That $S$ is symmetric and non-negative follows from the approximation of $\mathtt{C}_{ij}(x,y)$ by the coefficient $\E\left[(U_\alpha\chi_i)(\bar{\mathtt{m}}_0)(x)\chi_j(\bar{m_0})(y)\right]$ where $\alpha>0$.
Then the right-hand side of \eqref{eq:square} approximates $\dual{Su}{u}$ by a quantity (the left-hand side thus) which is non-negative and symmetric.

\end{proof}

Being an operator with kernel $\mathtt{C}\in L^\infty(\T^d;\R^d)\subset L^2(\T^d;\R^d)$, the operator $S$ is compact. Actually, $S$ is trace-class, \cite[Corollary~11A, p.344]{Lax2002}. By the spectral theorem, there exists an orthonormal basis $(p_k)_{k\in\N}$ of $L^2(\T^d;\R^d)$ and a sequence $(\mu_k)_{k\in\N}$ of non-negative real numbers such that $S=\sum_k\mu_k p_k\otimes p_k$. The square-root of $S$ is then the operator $S^{1/2}$ defined by
\begin{equation}\label{Shalf}
S^{1/2}=\sum_{k\in\N}\mu_k^{1/2} p_k\otimes p_k.
\end{equation}
In the spectral decomposition of $S$ and in \eqref{Shalf}, we use the notation $p_k\otimes p_k$ to denote the projection operator $u\mapsto\dual{u}{p_k}p_k$. Note, \cite[Theorem~11, p.343]{Lax2002}, that 
\begin{equation}\label{diagC}
C_{ij}(x,x)=\sum_{k\in\N}\mu_k p_{k,i}(x)p_{k,j}(x),
\end{equation}	
for all $x\in\T^d$ and $i,j\in\{1,\dotsc,d\}$. The identity \eqref{diagC} will be used in Section~\ref{sec:limeqSolve}. Eventually, we introduce a sequence $(\beta_k(t))_{k\in\N}$ of independent one-dimensional Wiener processes and define the  the cylindrical Wiener process on $L^2(\T^d;\R^d)$
\begin{equation}\label{defCylW}
W(t)=\sum_{k\in\N}\beta_k(t) p_k.
\end{equation}
The Wiener process $W(t)$ will be the driving noise in the limit equation solved in Section~\ref{sec:limeqSolve}.
We end this section with the proof that $K^*\geq K(M)$ in the reversible case.

\begin{proposition}[Enhanced diffusion] Let $(\bar{\mathtt{m}}_t)$ be an admissible pilot process in the sense of Definition~\ref{def:admbarm}. Assume that $(\bar{\mathtt{m}}_t)$ is reversible. Then the matrix $K^*$ defined by \eqref{Kflat} satisfies $K\geq K(M)$.
\label{prop:Kstar}\end{proposition}

\begin{proof}[Proof of Proposition~\ref{prop:Kstar}] Let $\xi\in\R^d$ and let $\theta(n)=\chi(n)\cdot\xi$. We have to prove that
\begin{equation}\label{posposK}
\dual{U_0U_1\theta}{\theta}_{L^2(\lambda)}=\E\left[(U_0U_1\theta)(\bar{\mathtt{m}}_0)\theta(\bar{\mathtt{m}}_0)\right]
\end{equation}
is non-negative. It is sufficient to prove $\dual{U_\alpha U_\beta\theta}{\theta}_{L^2(\lambda)}\geq 0$ for $\alpha,\beta>0$. By differentiation of the resolvent formula $U_{\alpha+h} U_\alpha=h^{-1}(U_\alpha-U_{\alpha+h})$, we obtain $\partial_\alpha U_\alpha=-U_\alpha^2$. Setting $\varphi(\alpha)=\dual{U_\alpha U_\beta\theta}{\theta}_{L^2(\lambda)}$, we deduce that 
\begin{equation}\label{posposK3}
\varphi'(\alpha)=-\dual{U^2_\alpha U_\beta\theta}{\theta}_{L^2(\lambda)}=-\dual{U_\beta U_\alpha \theta}{U_\alpha\theta}_{L^2(\lambda)}.
\end{equation}
To obtain the last identity in \eqref{posposK3}, we have used the fact that $U_\alpha$ is symmetric in $L^2(\lambda)$. It follows from \eqref{posposK3} and \eqref{eq:square} that $\varphi'(\alpha)\leq 0$.
We also have $\lim_{\alpha\to+\infty}\varphi(\alpha)=0$, therefore $\varphi(\alpha)\geq 0$ for all $\alpha>0$, which is the desired result.
\end{proof}

\section{Tightness}\label{sec:tightness}

In this section we establish various bounds on the solution $f^\eps$ to  \eqref{eq:1} to establish in particular that the associated sequence $(\rho^\eps)$ is tight in $C([0,T];H^{-\sigma}(\T^d))$, \textit{cf.} Proposition~\ref{prop:tight}. 

\subsection{Bound in \texorpdfstring{$L^1$}{}}\label{sec:L1bound}

The integral of the solution $f^\eps_t$ to \eqref{eq:1} over $\T^d\times V$ is constant in time. If $f^\eps_\mathrm{in}\geq 0$ a.e., then, almost surely, $f^\eps_t\geq 0$ a.e. by \eqref{BallR}, \eqref{Rsmall}, which ensure that the positivity hypothesis \eqref{ModifiedEqPos} is satisfied a.s. Consequently, under \eqref{Hypin}, we have: almost surely, for all $t\geq 0$,
\begin{equation}\label{L1OK}
\|f^\eps(t)\|_{L^1(\T^d\times V)}\leq C_\mathrm{in}.
\end{equation}

\subsection{Relative entropy estimate}\label{sec:relativeentropy}

Let $\bar{M}_t$ be defined by \eqref{def:barMt}. Let $\bar{M}^\eps_t=\bar{M}_{\eps^{-2}t}$. We consider the relative entropy
\begin{equation}\label{defHH}
\mathcal{H}^\eps(t):=\mathcal{H}(f^\eps(t)|\bar{M}^\eps_t):=\iint_{\T^d\times V}H\left(\frac{f^\eps(x,t,v)}{\bar{M}^\eps_t(x,v)}\right) \bar{M}^\eps_t(x,v)dx d\nu(v),
\end{equation}
where $H$ is the square function $H(u)=\frac{u^2}{2}$. As we will see in Proposition~\ref{prop:relent} below, the dissipation term associated to $\mathcal{H}^\eps(t)$ is, up to multiplicative factors,
\begin{equation}\label{entropydiss}
\mathcal{D}^\eps(t)=\iint_{\T^d\times V}\frac{|R(f^\eps_t)\bar{M}^\eps_t-f^\eps_t|^2}{\bar{M}^\eps_t}dxd\nu(v).
\end{equation}

\begin{proposition}[Relative entropy estimate] Let $f^\eps_\mathrm{in}\in L^2(\T^d\times V)$. Let $M$ satisfy \eqref{HypM}. Let $(\bar{\mathtt{m}}_t)$ be an admissible pilot process in the sense of De\-fi\-ni\-tion~\ref{def:admbarm}. Then the mild solution $f^\eps_t$ to \eqref{eq:1} with initial datum $f^\eps_\mathrm{in}$ satisfies the relative entropy estimate
\begin{equation}\label{relativeentropy}
\mathcal{H}^\eps(t)+\frac{1}{2\eps^2}\int_0^t\mathcal{D}^\eps(s) ds
\leq e^{t}\mathcal{H}^\eps(0),
\end{equation}
almost surely, for every $t\geq 0$.
\label{prop:relent}\end{proposition}

\begin{proof}[Proof of Proposition~\ref{prop:relent}] Introduce the operators and function
\begin{equation}\label{defLepstc}
L^\eps_t f= R(f)\bar{M}^\eps_t-f,\quad \check{L}^\eps_t f= R(f)\check{M}^\eps_t-f,\quad\check{M}^\eps_t=M+v\cdot\nabla_x\bar{\mathtt{m}}^\eps_t.
\end{equation}
Using first the equation~\eqref{eq:1} for $f^\eps_t$, which reads
$
\partial_t f^\eps_t+\eps^{-1}v\cdot\nabla_x f^\eps_t=\eps^{-2}\check{L}^\eps_t f^\eps_t,
$
using, secondly, the equation $\partial_t\bar{M}^\eps_t=\eps^{-2}\check{L}^\eps_t\bar{M}^\eps_t$ for the reference solution $\bar{M}^\eps_t$
and Proposition~\ref{prop:regmildLB} to justify the following computations, we obtain the decomposition
\begin{equation}\label{decH}
\frac{d\;}{dt}\mathcal{H}^\eps(t)=-\frac{1}{\eps^2}A^\eps_t+\frac{1}{\eps}B^\eps_t,
\end{equation}
where
\[
A^\eps_t=-\iint_{\T^d\times V}\left[\frac{f^\eps_t}{\bar{M}^\eps_t}\check{L}^\eps_t f^\eps_t-\frac12\frac{|f^\eps_t|^2}{|\bar{M}^\eps_t|^2}\check{L}^\eps_t\bar{M}^\eps_t\right] dxd\nu(v),
\]
and, after integration by parts,
\[
B^\eps_t=-\frac12\iint_{\T^d\times V}\frac{|f^\eps_t|^2}{|\bar{M}^\eps_t|^2}v\cdot\nabla_x\bar{M}^\eps_t dxd\nu(v).
\]
We use the identities
$
\check{L}^\eps_t f=L^\eps_t f+ R(f)(\check{M}^\eps_t-\bar{M}^\eps_t),\quad \check{L}^\eps_t\bar{M}^\eps_t=\check{M}^\eps_t-\bar{M}^\eps_t,
$
to write
\begin{equation}\label{HAeps1}
A^\eps_t=-\iint_{\T^d\times V}\frac{f^\eps_t}{\bar{M}^\eps_t} L^\eps_t f^\eps_t dxd\nu(v)+\iint_{\T^d\times V}\left[\frac12\frac{|f^\eps_t|^2}{|\bar{M}^\eps_t|^2}-\rho^\eps_t\frac{f^\eps_t}{\bar{M}^\eps_t}\right](\check{M}^\eps_t-\bar{M}^\eps_t) dxd\nu(v).
\end{equation}
We also have
\[
\iint_{\T^d\times V}\rho^\eps_t L^\eps_t f^\eps_t dxd\nu(v)=0,\quad \iint_{\T^d\times V}\rho^\eps_t (\check{M}^\eps_t-\bar{M}^\eps_t) dxd\nu(v)=0.
\]
The right-hand side of \eqref{HAeps1} is therefore
\begin{align*}
A^\eps_t&=\iint_{\T^d\times V}\frac{|L^\eps_t f^\eps_t|^2}{\bar{M}^\eps_t}  dxd\nu(v)+\iint_{\T^d\times V}\left[\frac12\frac{|f^\eps_t|^2}{|\bar{M}^\eps_t|^2}-\rho^\eps_t\frac{f^\eps_t}{\bar{M}^\eps_t}+\frac12|\rho^\eps_t|^2|\right](\check{M}^\eps_t-\bar{M}^\eps_t) dxd\nu(v)\\
&=\iint_{\T^d\times V}\frac{|L^\eps_t f^\eps_t|^2}{\bar{M}^\eps_t}  dxd\nu(v)+\iint_{\T^d\times V}\frac{|L^\eps_t f^\eps_t|^2}{\bar{M}^\eps_t} \frac{\check{M}^\eps_t-\bar{M}^\eps_t}{2\bar{M}^\eps_t} dxd\nu(v).
\end{align*}
It follows that
\begin{equation}\label{HAeps2}
A^\eps_t=\iint_{\T^d\times V}\frac{|L^\eps_t f^\eps_t|^2}{\bar{M}^\eps_t}\frac{\check{M}^\eps_t+\bar{M}^\eps_t}{\bar{M}^\eps_t}  dxd\nu(v)\geq\mathcal{D}^\eps(t).
\end{equation}
Indeed, $\check{M}^\eps_t\geq 0$ a.e. due to \eqref{HypM}, \eqref{BallR}, \eqref{Rsmall}. In the second term $B^\eps_t$, we can decompose $f^\eps_t=\rho^\eps_t\bar{M}^\eps_t-L^\eps_t f^\eps_t$. Since $|v\cdot\nabla_x\bar{\mathtt{m}}^\eps_t|\leq\mathtt{b}$ by \eqref{BallR} and $\alpha-\mathtt{b}\leq\bar{M}^\eps_t$ by \eqref{HypM}-\eqref{BallR}, we have the estimate
\[
\frac{1}{\eps}|B^\eps_t|\leq\frac{1}{\eps^2}\frac{\mathtt{b}}{\alpha-\mathtt{b}}\mathcal{D}^\eps(t)
+\int_{\T^d}|\rho^\eps_t|^2 dx.
\]
Under the smallness hypothesis \eqref{Rsmall}, we have $\frac{\mathtt{b}}{\alpha-\mathtt{b}}\leq\frac12$. Using \eqref{decH} and \eqref{HAeps2}, we deduce that
\begin{equation}\label{decH2}
\frac{d\;}{dt}\mathcal{H}^\eps(t)\leq-\frac{1}{2\eps^2}\mathcal{D}^\eps(t)+\int_{\T^d}|\rho^\eps_t|^2 dx.
\end{equation}
The Cauchy-Schwarz inequality applied to the product of $[\bar{M}^\eps_t]^{-1/2}f^\eps_t$ with $[\bar{M}^\eps_t]^{-1/2}$ gives also the inequality
\begin{equation}\label{L2rhofromH}
\int_{\T^d} |\rho^\eps_t|^2dx\leq\iint_{\T^d\times V} \frac{|f^\eps_t|^2}{\bar{M}^\eps_t}dx d\nu(v)=\mathcal{H}^\eps(t).
\end{equation}
It follows that $\frac{d\;}{dt}\mathcal{H}^\eps(t)\leq-\frac{1}{2\eps^2}\mathcal{D}^\eps(t)+\mathcal{H}^\eps(t)$. We apply the
Gr\"onwall Lemma and get \eqref{relativeentropy}. This concludes the proof.
\end{proof}

Assume that the uniform $L^2$-bound in \eqref{Hypin} is satisfied. Two important corollaries from \eqref{relativeentropy} are then, first, using \eqref{HypM}, the estimate \eqref{onLocalEquilibrium} and, second, using \eqref{L2rhofromH}, the bound
\begin{equation}\label{L2estimate}
\|\rho^\eps_t\|_{L^2(\T^d\times V)}^2\leq \alpha^{-2}C_\mathrm{in}^2,
\end{equation}
almost surely, which gives a uniform $L^2$-estimate on $\rho^\eps_t$. We use these bounds in the next Section~\ref{sec:timereg} to obtain the tightness of $(\rho^\eps_t)$ in the space $C([0,T];H^{-\sigma}(\T^d))$.

%
%
%

\subsection{Time regularity}\label{sec:timereg}

For $\sigma>0$, we denote by $H^{-\sigma}(\T^N)$ the dual space of $H^\sigma(\T^N)$. 

\begin{proposition} Let $f^\eps_\mathrm{in}\in L^2(\T^d\times V)$. Let $(\bar{\mathtt{m}}_t)$ be an admissible pilot process in the sense of De\-fi\-ni\-tion~\ref{def:admbarm}. Let $f^\eps$ be the mild solution to \eqref{eq:1} with initial datum $f^\eps_\mathrm{in}$. Assume that $f^\eps$ and $M$ satisfy \eqref{Hypin} and \eqref{HypM} respectively. Then $(\rho^\eps_t)_{t\in[0,T]}$ is tight in the space $C([0,T];H^{-\sigma}(\T^d))$.
\label{prop:tight}\end{proposition}

\begin{proof}[Proof of Proposition~\ref{prop:tight}]  The proof is very similar to the proof of Proposition~5.11 in \cite{DebusscheVovelle20}, so we will only give the main arguments. First, using the uniform estimate \eqref{L2estimate}, it is sufficient to establish the result for some $\sigma>1$. Given $\xi\in H^\sigma(\T^d)$, we are interested in getting estimates on the quantity $\dual{\rho^\eps_t}{\xi}$. To that purpose, we use the perturbed test-function method devised in the previous Section~\ref{sec:perturbed}. In view of the expression~\eqref{varphi1-2} of the first corrector, we introduce the perturbation
\begin{equation}\label{zetaeps}
\zeta^\eps_t=\rho^\eps_t-\eps\divv_x(J(f^\eps_t)+\rho^\eps_t U_0\chi(\bar{\mathtt{m}}^\eps_t)).
\end{equation}
We have then $\dual{\zeta^\eps_t}{\xi}=\varphi^\eps(f^\eps_t,\bar{\mathtt{m}}^\eps_t)$, where $\varphi(\rho)=\dual{\rho}{\xi}$ is of the form considered in Proposition~\ref{prop:varphiADM} with $\psi(s)=s$, $\varphi_1$ is the corresponding first corrector given by \eqref{varphi1-2} and $\varphi^\eps=\varphi+\eps\varphi_1$ is the first order perturbation of $\varphi$. By \eqref{L2estimate}, the difference $\rho^\eps-\zeta^\eps$ tends to $0$ in $C([0,T];H^{-\sigma}(\T^d))$. To conclude, it is therefore sufficient to show that $(\zeta^\eps)$ is tight in the Skorokhod space $D([0,T];H^{-\sigma}(\T^d))$ (we refer to the proof of Proposition~5.11 in \cite{DebusscheVovelle20} here, for the details of this argument). We apply Theorem~3.1 of \cite{Jakubowski86} that gives a criteria of tightness in some Skorokhod spaces of vector valued functions. Condition (3.3) in \cite[Theorem~(3.1)]{Jakubowski86} is trivially satisfied with $\eps=0$ and $K_\eps=L^2(\T^d)$ by \eqref{L2estimate}. To check the condition (3.3) in \cite[Theorem~(3.1)]{Jakubowski86} we are reduced to prove that, for all $\xi\in H^\sigma(\T^d)$, the sequence $(\dual{\zeta^\eps_t}{\xi})$ is tight in $D([0,T])$. To prove this, we apply the Aldous criterion, \cite[p.356]{JacodShiryaev03}. Let $\tau_1$, $\tau_2$ be some stopping times (for the filtration generated by $(\bar{\mathtt{m}}^\eps_t)$) satisfying $\tau_1<\tau_2<\tau_1+\theta$ and $\tau_2\leq T$ almost surely, where $\theta\in(0,1)$. We obtain an estimate on the increment 
\begin{equation}\label{incr1}
\E\left[|\dual{\zeta^\eps_{\tau_2}}{\xi}-\dual{\zeta^\eps_{\tau_1}}{\xi}|^2\right]=\E\left[|\varphi^\eps(f^\eps_{\tau_2},\bar{\mathtt{m}}^\eps_{\tau_2})-\varphi^\eps(f^\eps_{\tau_1},\bar{\mathtt{m}}^\eps_{\tau_1})|^2\right]
\end{equation}
as follows. We introduce the process
\begin{equation}\label{incr-M}
M^\eps_t=\varphi^\eps(f^\eps_{t},\bar{\mathtt{m}}^\eps_{t})-\varphi^\eps(f^\eps_{0},\bar{\mathtt{m}}^\eps_{0})-\int_0^t\LL^\eps\varphi^\eps(f^\eps_{s},\bar{\mathtt{m}}^\eps_{s})ds.
\end{equation}
By Proposition~\ref{prop:varphiADM} and Proposition~\ref{prop:first-corrector} we have $\LL^\eps\varphi^\eps=\LL^\flat\varphi_1$, which is bounded on bounded sets, uniformly with respect to $\eps$. Since $(f^\eps_{s},\bar{\mathtt{m}}^\eps_{s})$ is almost surely uniformly bounded in $\mathcal{X}$, \textit{cf.} \eqref{L1OK}, we can estimate \eqref{incr1} by considering the increments of $(M^\eps_t)$:
\begin{equation}\label{incr2}
\E\left[|\dual{\zeta^\eps_{\tau_2}}{\xi}-\dual{\zeta^\eps_{\tau_1}}{\xi}|^2\right]\leq 2\E\left[|M^\eps_{\tau_2}-M^\eps_{\tau_1}|^2\right]+C\theta^2,
\end{equation}
where $C$ depends on the constant $\mathtt{b}$ in \eqref{BallR} and on the constant $C_\mathrm{in}$ in \eqref{Hypin}. By Proposition~\ref{prop:first-corrector}, Proposition~\ref{prop:SquareOrderone}, and Appendix~\ref{sec:app3}, $(M^\eps_t)$ is a martingale, and the process
\[
Z^\eps_t=|M^\eps_t|^2-A^\eps_t,\quad A^\eps_t=\int_0^t\left[\LL^\eps|\varphi^\eps|^2-2\varphi^\eps\LL\varphi^\eps\right](f^\eps_{s},\bar{\mathtt{m}}^\eps_{s})ds,
\]
is a martingale. By the Doob's optimal sampling theorem, we deduce that
\[
\E\left[|M^\eps_{\tau_2}-M^\eps_{\tau_1}|^2\right]=\E\left[|M^\eps_{\tau_2}|^2-|M^\eps_{\tau_1}|^2\right]=\E\left[A^\eps_{\tau_2}-A^\eps_{\tau_1}\right].
\]
We apply then the result of Proposition~\ref{prop:SquareOrderone} to conclude that 
\begin{equation}\label{incr3}
\E\left[|M^\eps_{\tau_2}-M^\eps_{\tau_1}|^2\right]\leq C\theta,
\end{equation}
for a constant $C$ as in \eqref{incr2}. Then, by \eqref{incr2}, \eqref{incr3} and the Markov inequality, we obtain
\[
\lim_{\theta\to 0}\limsup_{\eps\in(0,1)}\sup_{\tau_1,\tau_2}\PP(|\dual{\zeta^\eps_{\tau_2}}{\xi}-\dual{\zeta^\eps_{\tau_1}}{\xi}|>\eta)=0
\]
for all $\eta>0$. This establishes the Aldous criterion and concludes the proof.
\end{proof}

\subsection{Convergence to the solution of a Martingale problem}\label{sec:cvDA}

Assume that the hypotheses of Proposition~\ref{prop:tight} are satisfied. Let $(\eps_n)$ be a sequence which decreases to $0$ and let $\sigma>0$. Set $\eps_\N=\{\eps_n;n\in\N\}$. By the Prohorov theorem \cite[p.~59]{BillingsleyBook}, there is a subset of $\eps_\N$, which we still denote by $\eps_\N$, 
and a random variable $\rho^0$ on $C([0,T];H^{-\sigma}(\T^d))$ such that $(\rho^{\eps})$ is converging to $\rho^0$ in law in $C([0,T];H^{-\sigma}(\T^d))$ along $\eps_\N$. By \eqref{L2estimate}, we know that: almost surely, for all $t\in[0,T]$, $\rho^0_t\in L^2(\T^d)$. 
Let $(\F_t)_{t\in[0,T]}$ be the natural filtration of $(\rho^0(t))_{t\in[0,T]}$. Our aim is to show that the process $(\rho^0(t))_{t\in[0,T]}$ is a solution of the martingale problem associated to the limit generator $\LL$.

\begin{proposition} Let $\sigma\in(0,1)$, $\xi\in C^\infty(\T^d)$, $\psi\in C^3(\R)$ and let $\varphi$ be defined by $\varphi(\rho)=\psi\left(\dual{\rho}{\xi}\right)$. Then the process 
\begin{equation}\label{Mvarphitilde}
M_\varphi(t):=\varphi(\rho^0(t))-\varphi(\rho^0(0))-\int_0^t\LL\varphi(\rho^0(s))ds
\end{equation}
is a continuous martingale with respect to $(\F_t)_{t\in[0,T]}$.
\label{prop:MartingalePbtilde}\end{proposition}
\begin{proof}[Proof of Proposition~\ref{prop:MartingalePbtilde}] Let $0\leq s\leq t\leq T$. Let $0\leq t_1<\cdots<t_n\leq s$ and let $\Theta$ be a continuous and bounded function on $[H^{-\sigma}(\T^d)]^n$. Note that $\F_s$ is generated by the random variables $\Theta(\rho^0(t_1),\ldots,\rho^0(t_n))$, for $n\in\N^*$, $(t_i)_{1,n}$ and $\Theta$ as above. Our aim is therefore to prove that
\begin{equation}\label{Martingaletilde}
\E\left[(M_\varphi(t)-M_\varphi(s))\Theta(\rho^0(t_1),\ldots,\rho^0(t_n))\right]=0.
\end{equation}
Let $\varphi^\eps=\varphi+\eps\varphi_1+\eps^2\varphi_2$ be the second order correction of $\varphi$, with $\varphi_1$ and $\varphi_2$ given by Proposition~\ref{prop:first-corrector} and Proposition~\ref{prop:second-corrector} respectively. The results of the appendix~\ref{sec:app3} ensure that
\begin{equation}\label{Mvarphitildeeps}
M^\eps_\varphi(t):=\varphi^\eps(f^\eps(t),\bar{\mathtt{m}}^\eps_t)-\varphi^\eps(f_\mathrm{in},\bar{\mathtt{m}}^\eps_0)-\int_0^t\LL^\eps\varphi^\eps(f^\eps(s),\bar{\mathtt{m}}^\eps_s)ds
\end{equation}
is a martingale. The identity
\begin{equation}\label{Martingaletildeeps}
\E\left[(M^\eps_\varphi(t)-M^\eps_\varphi(s))\Theta(\rho^\eps(t_1),\ldots,\rho^\eps(t_n))\right]=0,
\end{equation}
is therefore satisfied for all $\eps>0$. By Proposition~\ref{prop:first-corrector} and Proposition~\ref{prop:second-corrector} and the fact that $\LL^\eps\varphi^\eps=\LL\varphi+\eps\LL^\flat\varphi_2$, we deduce from \eqref{Martingaletildeeps} that
\begin{equation}\label{MartingaleX}
\E\left[(X^\eps_\varphi(t)-X^\eps_\varphi(s))\Theta(\rho^\eps(t_1),\ldots,\rho^\eps(t_n))\right]=\mathcal{O}(\eps),
\end{equation}
where the process $(X^\eps_\varphi(t))$ is 
\[
X^\eps_\varphi(t)=\varphi(\rho^\eps(t))-\varphi(\rho_\mathrm{in})-\int_0^t\LL\varphi(\rho^\eps(s))ds.
\]
The bound \eqref{L2estimate}, which is satisfied almost surely, allows us to consider the functions of $\rho$ involved in the definition of $M_\varphi(t)$ as bounded continuous functions on $C([0,T];H^{-\sigma}(\T^d))$. We can then take the limit in \eqref{MartingaleX} to conclude to \eqref{Mvarphitilde}.
\end{proof}

We apply Proposition~\ref{prop:MartingalePbtilde} with, successively, $\varphi(\rho)=\dual{\rho}{\xi}$, $\varphi(\rho)=|\dual{\rho}{\xi}|^2$ and use \eqref{SecondQuad} to obtain the following result.

\begin{corollary} Let $K^*$ and $\Psi$ be defined by \eqref{Kflat} and \eqref{coeffPsi}. Then the process 
\[
M_t^0:=\rho^0_t-\rho^0_\mathrm{in}-\int_0^t \divv(K^*\nabla_x\rho^0_s+\Psi\rho^0_s) ds
\]
is a continuous $H^{-2}(\T^d)$-valued martingale with quadratic variation given by 
\[
[M_t^0,M_t^0]\cdot(\xi,\xi)=2 \int_0^t \|S^{1/2}(\rho^0_s\nabla_x\xi)\|^2_{L^2(\T^d;\R^d)} ds,
\]
for all $\xi\in C^1(\T^d)$.
\label{cor:MartingalePbtilde}\end{corollary}

The details of the proof can be found in Section~5.5.2 of \cite{DebusscheVovelle20}. We refer also to the same paragraph of \cite{DebusscheVovelle20} for the application of the L\'evy representation theorem, \cite[Theorem~8.2, p.222]{DaPratoZabczyk14}, that gives the existence of a stochastic basis $(\tilde{\Omega},\tilde{\F},\tilde{\PP},(\tilde{\F}_t)_{t\geq 0},(\tilde{\beta}_k)_{k\in\N})$ and of an adapted process $H^{-2}(\T^d)$-valued $\tilde{\rho}^0$ with the same law as $\rho^0$, such that, $\tilde{\PP}$-almost surely, for all $t\in[0,T]$, for all $\xi\in H^2(\T^d)$ the identity
\begin{equation}\label{soltilde}
\dual{\tilde{\rho}^0_t}{\xi}=\dual{\rho^0_\mathrm{in}}{\xi}+\int_0^t \dual{\tilde\rho^0_s}{\divv(K^*\nabla_x\xi)-\Psi\cdot\nabla_x\xi)} ds-\sqrt{2}\sum_{k\geq 0}\int_0^t\dual{\tilde\rho^0_s}{\mu_k^{1/2}p_k\cdot\nabla_x\xi} d\tilde{\beta}_k(s)
\end{equation}
is satisfied. 
\section{The limit equation}\label{sec:limeq}

The analysis of the limit generator $\LL$ was initiated in Section~\ref{sec:arrange}. We complete this study here, in particular we establish the relation between $\LL$ and the stochastic PDE \eqref{eq:limeq}, and use the good properties of the limit equation \eqref{eq:limeq} to conclude the proof of convergence of $f^\eps$.

\subsection{Resolution of the limit equation}\label{sec:limeqSolve}

Let $S^{1/2}$ and $W(t)$ be defined by \eqref{Shalf} and \eqref{defCylW} respectively. There are several possible approaches to the resolution of the limit equation \eqref{eq:limeq}, \textit{e.g.} \cite{DaPratoZabczyk14} based on semi-group approach, \cite{PrevotRockner2007} based on Gelfand triples and monotony. Here we follow\footnote{the only drawback of this approach is that we possibly require an excessive regularity of the initial datum $\rho_\mathrm{in}$ and of the pilot process $(\bar{\mathtt{m}}_t)$} the recent reference \cite{DuLiu2019}, for the reason that it is immediate to cast \eqref{eq:limeq} into the framework of \cite{DuLiu2019} and that we can then directly apply the results of \cite{DuMeng2010} on backward parabolic SPDEs to obtain a satisfactory uniqueness result (see Section~\ref{sec:CL}). To proceed, we observe first that \eqref{eq:limeq} can be written under the following form (where we use summation over repeated indexes)
\begin{equation}\label{limeqDuLiu}
d\rho^0=(a^{ij}\partial^2_{x_i x_j}\rho^0+b^i \partial_{x_i} \rho^0+c\rho^0)dt+(\sigma^{ik}\partial_{x_i}\rho^0+\nu^k\rho^0)d\beta^k(t),
\end{equation}
with
\begin{equation}\label{CoeffDuLiu}
a^{ij}=K^*_{ij},\; b^i=\partial_{x_i}(K^*_{ij})+\Psi_i,\; c=\partial_{x_i}(\Psi_i),\; \sigma^{ik}=\sqrt{2}\mu_k^{1/2}p_{k,i},\; \nu^k=\sqrt{2}\partial_{x_i}(\mu_k^{1/2}p_{k,i})
\end{equation} (we refer to Section~\ref{sec:arrange} for the definition of $\mu_k$ and $p_k$). Let us set $q_k=\mu_k^{1/2}p_k$. We have the following elementary result.

\begin{lemma}\label{lemma:l2coeff} Let $s\in\N$. The sequence $q_0,q_1,\dotsc$ satisfies the bound
\begin{equation}\label{eq:l2coeff}
\sum_{k\in\N}\|q_k\|^2_{H^s(\T^d;\R^d)}\leq \E\left[\|U_0(\bar{\mathtt{m}}_0)\|_{H^{2s}(\T^d;\R^d)}^2
+\|\bar{\mathtt{m}}_0\|_{H^{2s}(\T^d;\R^d)}^2\right].
\end{equation}
\end{lemma}

By \eqref{BallR} and \eqref{mixCoupled} and the Sobolev injection $H^s(\T^d;\R^d)\hookrightarrow C^{2+\delta}(\T^d;\R^d)$ for $s>2+d/2$, where $\delta$ is a certain positive number, we deduce that $\sigma^i\in\ell^2(\N;C^{1+\delta}(\T^d))$, $\nu\in\ell^2(\N;C^{1+\delta}(\T^d))$. The proof of Lemma~\ref{lemma:l2coeff} is similar to the proof of Proposition~5.14 in \cite{DebusscheVovelle20}, so we skip the details. By \eqref{BallR} and \eqref{mixCoupled}, we also have that the $C^{\delta}(\T^d)$ norm of the coefficients $a^{ij}$, $b^i$, $c$ is finite. Condition \textbf{(H)} in \cite[p.2647]{DuLiu2019} is fulfilled therefore. We also need to check that \eqref{limeqDuLiu} is parabolic in the following sense, \textit{cf.} Condition~(1.2) in \cite{DuLiu2019}: there exists $\lambda_0>0$ such that 
\begin{equation}\label{DuLiu12}
\lambda_0\mathrm{I}_d+\sigma\sigma^*\leq 2a\leq\lambda_0^{-1}\mathrm{I}_d.
\end{equation}
Clearly, only the first inequality in \eqref{DuLiu12} is at stake here. By \eqref{SKernelK} and \eqref{diagC}, we have
\[
(\sigma\sigma^*)(x)=2\E\left[(U_0\chi)(\bar{\mathtt{m}}_0)(x)\otimes\chi(\bar{m_0})(x)\right].
\]
From the definition \eqref{Kflat} of $K^*$ and the resolvent formula $U_0U_1=U_0-U_1$, we deduce that
\begin{equation}\label{2asigma}
2a(x)-(\sigma\sigma^*)(x)=2K(M)+2\E\left[(U_1\chi)(\bar{\mathtt{m}}_0)(x)\otimes\chi(\bar{m_0})(x)\right].
\end{equation}
By \eqref{diagC}, the matrix $\mathtt{C}(x,x)$ is non-negative. If we replace $U_0$ by $U_1$ in the definition of $\mathtt{C}$ in \eqref{SKernelK}, we obtain a new operator $S$ with the same properties as the former operator. In particular, the last term in \eqref{2asigma} is a non-negative matrix. The matrix $K(M)$ is strictly positive. Indeed, if $\xi\in\R^d$, then $K(M)\xi\cdot\xi$ is the integral against $\nu$ of $v\mapsto|v\cdot\xi|^2$. If this quantity vanishes, then $v\cdot\xi=0$ for $\nu$-a.e. $v\in V$ since $M\geq\alpha$ by \eqref{HypM}. By \eqref{nuND}, $\xi=0$. This gives \eqref{DuLiu12} for a certain positive $\lambda_0$. We can now apply \cite[Corollary~1.2]{DuLiu2019}, that asserts that \eqref{eq:limeq} with initial datum $\rho_\mathrm{in}$ has a unique classical solution $\rho^0$, in the sense that the predictable process $\rho^0$ has the properties $\rho^0_t\colon\T^d\to L^p(\Omega)$ for all $t$, where $p$ is a given exponent greater than $2$, $\rho(\cdot,t,,\omega)$ is of class $C^2$ for all $(t,\omega)$, while \eqref{limeqDuLiu} is satisfied at all points $x\in\T^d$, in the usual integrated form of SDE.

\subsection{Conclusion}\label{sec:CL}

We consider the solution $\tilde{\rho}^0$ to \eqref{eq:limeq} obtained in \eqref{soltilde}. For simplicity, we may remove the tilde from notations now. Let $\rho^*$ be the solution to \eqref{eq:limeq} given in the previous Section~\ref{sec:limeqSolve}. We want to show that $\rho^0=\rho^*$. This requires an adequate estimate on the difference $\rho^0-\rho^*$. By linearity of \eqref{eq:limeq}, the problem is to prove that, although $\rho^0$ has only a very limited regularity a priori, it is indeed the trivial solution when $\rho_\mathrm{in}=0$. A possible approach, to establish this, is to prove an energy estimate for such weak solutions. This follows directly from an It\^o formula, but the weak regularity of $\rho^0$ is precisely an obstacle here. One may try to regularize the equation, or apply \eqref{soltilde} with $\xi$ a member of the spectral basis of an adequate operator. However, there are two elliptic operators at skate here, $a:D^2_x$ and $\frac12\sigma\sigma^*:D^2_x$ (\textit{cf.} \eqref{CoeffDuLiu}). Even if the competition is favourable by \eqref{DuLiu12}, we do not know how to get the desired energy estimate. We will use instead duality and the results of \cite{DuMeng2010} on backward parabolic SPDEs.
The main idea, symbolically, is simple: if $(X_t)$ is a solution of the SDE $dX_t=AX_tdt+BX_t dW_t$ and $(Y_t,q_t)$ is a solution of the backward SDE $dY_t=-A^*Y_tdt-B^*q_t dt +q_t dW_t$, then $t\mapsto \E\dual{X_t}{Y_t}$ is constant by the It\^o formula. Here we use the fact that the backward SPDE
\begin{equation}\label{BSPDE}
d\psi_t=\left[-\partial^2_{x_i x_j}(a^{ij}\psi_t)+\partial_{x_i}(b^i \psi_t)-c\psi_t+\partial_{x_i}(\sigma^{ik}q^k_t)-\nu^k q^k_t\right]dt+ q^k_td\beta^k(t),
\end{equation}
with terminal condition $\psi_T=h$, where $h$ is some given $\F_T$-measurable function in $C^\infty(\T^d)$, admits some solutions of class $C^2$ to justify by regularization the It\^o formula that leads to the identity $\E\dual{\rho^0_T}{h}=\E\dual{\rho^0_\mathrm{in}}{\psi_0}=0$ and to the desired conclusion $\rho^0_T\equiv 0$. The details of this regularization procedure can be found in \cite[Section~5.5.3]{DebusscheVovelle20}.

\section{Appendix: Markov processes}\label{sec:appendix}

\subsection{Markov processes}\label{sec:app1}

If $E$ is a Polish space, we denote by ${\mathrm{BM}(E)}$ the Banach space of bounded Borel-measurable functions on $E$ with the sup-norm
$$
\|\varphi\|_{{\mathrm{BM}(E)}}=\sup_{x\in E}|\varphi(x)|.
$$
The set ${\mathrm{BC}(E)}$ is the subspace of continuous bounded functions. By Markov process, we mean the triplet constituted of a Markov semi-group, some probability kernels, and the associated Markov processes. More precisely, we suppose first that we are given a Markov semi-group $\mathbf{P}=(P_t)_{t\geq 0}$, which is defined a priori as a an endomorphism of the space $\mathrm{BM}(E)$ and satisfies the initial condition $P_0=\mathrm{Id}$, the semi-group property $P_t\circ P_s=P_{t+s}$ for $t,s\geq 0$, the preservation of positivity $P_t\varphi\geq 0$ when $\varphi\geq 0$, and fixes the constant function $\mathbf{1}$ equal to $1$ everywhere: $P_t\mathbf{1}=\mathbf{1}$ for all $t\geq 0$. The second element of our set of data is a probability kernel $Q(t,x,B)$: for all $\varphi\in\mathrm{BM}(E)$, for all $x\in E$,
\begin{equation}\label{probaKernel}
P_t\varphi(x)=\int_E \varphi(y) Q(t,x,dy),
\end{equation} 
where, for every $t\geq 0$, for every $x\in E$, $Q(t,x,\cdot)$ is a probability measure and the dependence in $x$ is measurable, in the sense that the right-hand side of \eqref{probaKernel} is a measurable function of $x$. 
The third and last element of our set of data is the set $\mathbf{X}=\left\{(X^x_t)_{t\geq 0};x\in E\right\}$ of Markov processes indexed by their starting points $x$: $X^x_0=x$ almost surely. The finite-dimensional distribution of $(X^x_t)_{t\geq 0}$ is given by
\begin{multline*}
\PP(X^x_{0}\in B_0,X^x_{t_1}\in B_1,\dotsc,X^x_{t_k}\in B_k)=\int_{B_0}\dotsb\int_{B_{k-1}}Q(t_{k}-t_{k-1},y_{k-1},B_k)\\
\times Q(t_{k-1}-t_{k-2},y_{k-2},dy_{k-1})\dotsb Q(t_1,y_0,dy_1)\delta_x(dy_0),
\end{multline*}
where $0\leq t_1\leq\dotsb\leq t_k$ and $B_0,\dotsc,B_k\in\G$ and the Markov property
\begin{equation}\label{appMarkovPty}
\E\left[\varphi(X^x_{t+s})|\F^{X}_t\right]=P_s\varphi(X^x_t)
\end{equation}
is satisfied for all $s,t\geq 0$, $\varphi\in\mathrm{BM}(E)$, where $(\F^X_t)=(\sigma(X^x(s)_{0\leq s\leq t}))$ is the filtration generated by $X^x$. There is a certain redundancy in the description above since the existence of probability kernels satisfying \eqref{probaKernel} can be deduced from the properties of $(P_t)_{t\geq 0}$, \cite[Proposition~1.2.3]{BakryGentilLedoux14}, while the construction of a Markov process with the required finite dimensional distribution is established in \cite[Theorem~1.1 p.157]{EthierKurtz86} for example. It is not limiting, however, to assume that all these elements are given altogether, all the more since the processes $(X^x_t)_{t\geq 0}$ will generally have additional pathwise properties, being typically continuous or c{\`a}dl{\`a}g. They may also satisfy the Markov property \eqref{appMarkovPty} with respect to a given filtration $(\F_t)$ larger than $(\F^X_t)$.

\subsection{Infinitesimal generator}\label{sec:app2}

Given a Markov process as in Section~\ref{sec:app1}, we would like to define the associated infinitesimal generator. There are various possible approaches. In \cite{BakryGentilLedoux14} for example, it is assumed that the process admits an invariant measure $\mu$. The semi-group can then be extended as a contraction semi-group on $L^2(\mu)$. By assuming additionally that this extension gives rise to a strongly continuous semi-group, \cite[Property (vi), p.11]{BakryGentilLedoux14}, one can use the standard theory of strongly continuous semi-group, \cite{Pazy1983}, to define the infinitesimal generator. We may follow this approach, since the Markov processes that we are considering have all invariant measures. It would require some additional material, for instance the proofs that the Markov process $(f_t,\mathtt{m}(t;n))_{t\geq 0}$ of Theorem~\ref{th:MarkovPty} and that the limiting Markov process $(\rho^0(t))_{t\geq 0}$ of Theorem~\ref{th:mainth} have invariant measures. These results have a definite interest, but would divert us from our aim, the proof of Theorem~\ref{th:mainth}. To reach this aim more directly, we define the infinitesimal generator by convergence of the time increments under $\pi$-convergence. We say that a sequence $(\varphi_n)$ of $\mathrm{BM}(E)$ is $\pi$-converging to $\varphi\in\mathrm{BM}(E)$ (denoted $\varphi_n\topi\varphi$) if $\sup_n\|\varphi_n\|_{\mathrm{BM}(E)}<+\infty$ and $\varphi_n(x)\to\varphi(x)$ for all $x\in E$. This denomination is introduced in \cite{Priola99} and coincides with the terminology of bounded pointwise convergence (b.c.p.) used in \cite[p.111]{EthierKurtz86}. A function $\varphi\in\mathrm{BM}(E)$ is in the domain $D(\LL)$ of the infinitesimal generator $\LL$ of $(P_t)$ if there exists $\psi\in\mathrm{BM}(E)$ such that 
\begin{equation}\label{psiDL}
\frac{P_t\varphi-\varphi}{t}\topi\psi,
\end{equation}
when $t\to 0$. We then set $\LL\varphi=\psi$. Note that, on the elements $\varphi\in D(\LL)$, the property of continuity
\begin{equation}\label{pisemigroup}
P_t\varphi\topi\varphi
\end{equation}
when $t\to 0$ is satisfied. By the semi-group property, \eqref{pisemigroup} implies the property of continuity from the right $P_t\varphi\topi P_{t_*}\varphi$ when $t\downarrow t_*$, for every $t_*\geq 0$.

\subsection{Martingale property of Markov processes}\label{sec:app3}

Consider a Markov process as in Section~\ref{sec:app1}, which is Markov with respect to a filtration $(\F_t)$, and has a generator $\LL$, as defined as in Section~\ref{sec:app2}. We make the following hypotheses:
\begin{enumerate}
\item\label{item-stocont}\emph{stochastic continuity:} we have $P_t\varphi\topi P_{t_*}\varphi$ when $t\to t_*$ for every $\varphi\in\mathrm{BC}(E)$ and every $t_*\geq 0$,
\item\label{item-meas}\emph{measurability:} for all $x\in E$, the application $(\omega,t)\mapsto X^x_t(\omega)$ is measurable $\Omega\times\R_+\to E$.
\end{enumerate}
Then, for all $\varphi$ in the domain of $\LL$, for all $x\in E$,
\begin{equation}\label{XMartingaleE}
M^x_t:=\varphi(X^x_t)-\varphi(x)-\int_0^t\LL\varphi(X^x_s)ds
\end{equation}
is a $(\F_t)$-martingale. If furthermore $|\varphi|^2$ is in the domain of $\LL$, then the process $(Z^x_t)$ defined by 
\begin{equation}\label{VarXMartingaleE}
Z^x_t:=|M^x_t|^2-\int_0^t (\LL|\varphi|^2-2\varphi\LL\varphi)(X^x_s)ds,
\end{equation}
is a martingale.

\begin{remark} The proof of these results can be found in the appendix to \cite{DebusscheVovelle20}.
\label{rk:stocont0}\end{remark} 

\begin{remark}Since $\|P_t\varphi\|_{\mathrm{BM}(E)}\leq\|\varphi\|_{\mathrm{BM}(E)}$, the convergence $P_t\varphi\topi P_{t_*}\varphi$ is equivalent to the convergence $P_t\varphi(x)\to P_{t_*}\varphi(x)$ for all $x\in E$. The hypothesis of stochastic continuity can therefore be rephrased as the continuity, for the topology of the weak convergence of probability measures, of $t\mapsto P_t^*\delta_x$, where $P_t^*\delta_x$ is the law of $X^x_t$. This continuity property is satisfied if $(X^x_t)$ is stochastically continuous in particular: for all $\delta>0$,
\begin{equation}\label{eq:stocont}
\lim_{t\to t_*}\PP(d_E(X^x_t,X^x_{t_*})>\delta)=0,
\end{equation}
where $d_E$ is the distance on $E$. Under \eqref{eq:stocont}, we can also assume, up to a modification of the process, that the measurability property of Item~\ref{item-meas} is satisfied, \cite[Proposition~3.2]{DaPratoZabczyk14}.
\label{rk:stocont}\end{remark} 

\begin{remark} Assume that $E$ is a Banach space and that the process $(\mathbf{X}_t)$ is locally bounded in the following sense: for all $r_0>0$, for all $T>0$, there exists $r_T>0$ such that: $\PP$-almost surely, for all $x$ in the ball $\bar{B}(0,r_0)$ of $E$, for all $t\in[0,T]$, $X_t^x\in\bar{B}(0,r_T)$. Then all the above assertions can be extended to the cases where the functions $\varphi$ are measurable (respectively, continuous; respectively, Lipschitz continuous) and bounded on bounded sets. Indeed, as long as the set of initial values $x$ is restricted to a bounded set $\bar{B}(0,r_0)$, and time is restricted to a finite interval $[0,T]$, we have $P_t\varphi(x)=P_t(\varphi\circ \theta_{r_T}) (x)$, where, given $r>0$, $\varphi\circ\theta_{r_T}$ is the measurable (respectively, continuous; respectively, Lipschitz continuous) bounded function, obtained by composition of $\varphi$ with the truncation operator $\theta_r(x)$ equal to $x$ if $\|x\|_E<r$, and to $rx/\|x\|_E$ otherwise.
\label{rk:item-locbound}\end{remark} 

\subsection{Change of coordinates}\label{sec:app4}

Let $H\colon E\to E^\prime$ be a continuous, bijective map from $E$ onto an other Polish space $E^\prime$. We consider the process $Y_t=H(X_t)$. This \emph{change of coordinates}~\cite[Section~1.15.1]{BakryGentilLedoux14} provides a new Markov process $Y=(Y_t)$ with probability kernels given by
\[
P^Y(t,y,B)=H_*P^X(t,y,B):=P^X(t,H^{-1}(y),H^{-1}(B)),
\]
for all $t\geq 0$, $y\in E^\prime$, $B$ measurable in $E^\prime$. If $X$ has the invariant measure $\lambda^X$, then $Y$ has the invariant measure $\lambda^Y$ given by $\lambda^Y(B)=H_*\lambda^X(B):=\lambda^X(H^{-1}(B))$. The state space of $Y$ is $E^\prime$. By composition from the right with $H$ or its inverse, we have some correspondences between the spaces $\mathrm{BM}(E)$ and $\mathrm{BM}(E^\prime)$, and $\mathrm{BC}(E)$ and $\mathrm{BC}(E^\prime)$. We have also
\[
P^Y_t\varphi=(P^X_t\varphi\circ H)\circ H^{-1},\quad \LL^Y\varphi=(\LL^X\varphi\circ H)\circ H^{-1},
\]
with, here again, an isomorphism between $D(\LL^Y)$ and $D(\LL^X)$ given by the composition from the right with $H$.

\subsection{Poisson's Equation}\label{sec:app5}

In this part, $E$ is a Banach space. We consider a Markov process $(\mathbf{X}_t)$ that is locally bounded (as described in Remark~\ref{rk:item-locbound} above) uniformly in time: for all $r_0>0$, there exists $r_1>0$ such that: $\PP$-almost surely, for all $x$ in the ball $\bar{B}(0,r_0)$ of $E$, for all $t\geq 0$, $X_t^x\in\bar{B}(0,r_1)$. We assume that the process has an invariant measure $\mu$ which has a bounded support and that it satisfies the following mixing hypothesis: there exists a continuous function $\Gamma_\mathrm{mix}\in L^1(\R_+)$ such that: for all $x,y\in \bar{B}(0,r_0)$, there is a constant $C(r_1)$ depending on $r_1$, and there is a coupling $(\hat{X}^x_t,\hat{X}^y_t)_{t\geq 0}$ of $(X^x_t,X^y_t)_{t\geq 0}$ such that 
\begin{equation}\label{coupling-appendix}
\E\left[d_E(\hat{X}^x_t,\hat{X}^y_t)\right]\leq C(r_1)\Gamma_\mathrm{mix}(t).
\end{equation}
We also assume that the property of stochastic continuity described in Item~\ref{item-stocont} of Section~\ref{sec:app3} is satisfied. We have then the following proposition.

\begin{proposition}[Poisson's equation] Let $\Phi\colon E\to\R$ be Lipschitz continuous on bounded sets. The Poisson equation $\LL\Psi=\Phi$ has a a solution $\Psi$ in the class of measurable functions bounded on bounded sets if, and only if, $\dual{\Phi}{\mu}=0$. In the latter case, a solution is given by the resolvent formula 
\begin{equation}\label{invPoisson}
\Psi(x)=-\int_0^\infty P_t\Phi(x)dt,
\end{equation}
and this solution is unique up to addition of a constant.
 \label{prop:Poisson}\end{proposition}

\begin{proof}[Proof of Proposition~\ref{prop:Poisson}] Let $\Phi$ be Lipschitz continuous on bounded sets. Let us assume that $\dual{\Phi}{\mu}=0$. We will show that \eqref{invPoisson} provides a solution of the Poisson equation. For $x\in E$, we fix $r_0>0$ large enough to contain both $x$ and the support of $\mu$. Since $\mu$ is invariant, we have 
\[
\dual{\Phi}{\mu}=\dual{P_t\Phi}{\mu}=\int_E \E\Phi(X^y_t)d\mu(y).
\]
It follows that
\begin{equation}\label{invPoisson1}
P_t\Phi(x)=P_t\Phi(x)-\dual{\Phi}{\mu}=\int_{\bar{B}(0,r_0)}\E\left[\Phi(X^x_t)-\Phi(X^y_t)\right] d\mu(y).
\end{equation}
We denote by $\Lip(\Phi,r_1)$ the Lipschitz constant of $\Phi$ on the ball $\bar{B}(0,r_1)$. We deduce from \eqref{invPoisson1} that
\[
|P_t\Phi(x)|\leq \Lip(\Phi,r_1)\int_{\bar{B}(0,r_0)}\E\left[d_E(X^x_t,X^y_t)\right] d\mu(y)=\Lip(\Phi,r_1)\int_{\bar{B}(0,r_0)}\E\left[d_E(\hat{X}^x_t,\hat{X}^y_t)\right] d\mu(y).
\]
Using the bound \eqref{coupling-appendix}, we obtain, up to a constant depending on $\Phi$ and $r_1$, a bound from above on $|P_t\Phi(x)|$ by the integrable function $\Gamma_\mathrm{mix}(t)$:
\begin{equation}\label{invPoisson2}
|P_t\Phi(x)|\leq C^\prime(r_1)\Gamma_\mathrm{mix}(t).
\end{equation} 
This shows that \eqref{invPoisson} defines a measurable function bounded on bounded sets. Using \eqref{invPoisson2}, we can also approximate $\Psi$ by the integral in \eqref{invPoisson} restricted to a finite time interval to justify the formula
\begin{equation}\label{invPoisson3}
\delta_t\Psi(x)=\frac{P_t\Psi(x)-\Psi(x)}{t}=\int_0^1 P_{st}\Phi(x)ds.
\end{equation}
From \eqref{invPoisson3} and the property of stochastic continuity, it is immediate to conclude that $\Psi\in D(\LL)$ and $\LL\Psi=\Phi$. Reciprocally, assume that the equation $\LL\Psi=\Phi$ has a solution in the class of measurable functions bounded on bounded sets. By the definition of $\LL$, we have
\begin{equation}\label{intLL}
P_t\Psi(x)=\Psi(x)+\int_0^t P_s\LL\Psi(x)ds,
\end{equation}
for all $t\geq 0$ and all $x\in E$. Using $\LL\Psi=\Phi$, integrating each member of the identity \eqref{intLL} against the invariant measure $\mu$, we obtain $t\dual{\Phi}{\mu}=0$, which establishes the necessary condition $\dual{\Phi}{\mu}=0$. At last, let us prove the uniqueness up to constant. First, we note that \eqref{invPoisson2} gives $|P_t\Phi(x)-\dual{\Phi}{\mu}|\leq C^\prime(r_1)\Gamma_\mathrm{mix}(t)$ and thus $P_t\Phi(x)\to\dual{\Phi}{\mu}$ for all $x$. If $\LL\Psi=0$, we can pass to the limit $t\to+\infty$ in \eqref{intLL} to obtain that $\Psi$ is a constant function.
\end{proof}

\section*{Acknowledgement}
This work was performed within the framework of the LABEX MILYON (ANR-10- LABX-0070) of Universit\'e de Lyon, within the program ``Investissements d'Avenir'' (ANR-11-IDEX-0007) operated by the French National Research Agency (ANR). The partial support of ANR-12-BS01-0019 - STAB - \emph{Stabilit\'e du comportement asymptotique d'EDP, de processus stochastiques et de leurs discr\'etisations} is also acknowledged.

\bibliographystyle{abbrv}
\bibliography{../../../../../nosbiblab-utf8}

\def\cprime{$'$}
\begin{thebibliography}{10}

\bibitem{BakryGentilLedoux14}
D.~Bakry, I.~Gentil, and M.~Ledoux.
\newblock {\em Analysis and geometry of {M}arkov diffusion operators}, volume
  348.
\newblock Springer, Cham, 2014.

\bibitem{BalGu15}
G.~Bal and Y.~Gu.
\newblock Limiting models for equations with large random potential: a review.
\newblock {\em Commun. Math. Sci.}, 13(3):729--748, 2015.

\bibitem{BillingsleyBook}
P.~Billingsley.
\newblock {\em Convergence of probability measures}.
\newblock Wiley Series in Probability and Statistics: Probability and
  Statistics. John Wiley \& Sons Inc., New York, second edition, 1999.
\newblock A Wiley-Interscience Publication.

\bibitem{ChalubMarkowichPerthameSchmeiser04}
F.~Chalub, P.~A. Markowich, B.~Perthame, and C.~Schmeiser.
\newblock Kinetic models for chemotaxis and their drift-diffusion limits.
\newblock {\em Monatsh. Math.}, 142(1-2):123--141, 2004.

\bibitem{DaPratoZabczyk14}
G.~Da~Prato and J.~Zabczyk.
\newblock {\em Stochastic equations in infinite dimensions}, volume 152 of {\em
  Encyclopedia of Mathematics and its Applications}.
\newblock Cambridge University Press, Cambridge, second edition, 2014.

\bibitem{DeBouardDebussche10}
A.~de~Bouard and A.~Debussche.
\newblock The nonlinear {S}chr\"odinger equation with white noise dispersion.
\newblock {\em J. Funct. Anal.}, 259(5):1300--1321, 2010.

\bibitem{DeBouardGazeau12}
A.~de~Bouard and M.~Gazeau.
\newblock A diffusion approximation theorem for a nonlinear {PDE} with
  application to random birefringent optical fibers.
\newblock {\em Ann. Appl. Probab.}, 22(6):2460--2504, 2012.

\bibitem{DebusscheDeMoorVovelle16}
A.~Debussche, S.~De~Moor, and J.~Vovelle.
\newblock Diffusion limit for the radiative transfer equation perturbed by a
  {M}arkovian process.
\newblock {\em Asymptot. Anal.}, 98(1-2):31--58, 2016.

\bibitem{DebusscheTsutsumi11}
A.~Debussche and Y.~Tsutsumi.
\newblock 1{D} quintic nonlinear {S}chr\"odinger equation with white noise
  dispersion.
\newblock {\em J. Math. Pures Appl. (9)}, 96(4):363--376, 2011.

\bibitem{DebusscheVovelle12}
A.~Debussche and J.~Vovelle.
\newblock Diffusion limit for a stochastic kinetic problem.
\newblock {\em Commun. Pure Appl. Anal}, 11(6):2305--2326, 2012.

\bibitem{DebusscheVovelle20}
A.~Debussche and J.~Vovelle.
\newblock Diffusion-approximation in stochastically forced kinetic equations.
\newblock {\em Tunis. J. Math.}, 3(1):1--53, 2021.

\bibitem{DegondGoudonPoupaud00}
P.~Degond, T.~Goudon, and F.~Poupaud.
\newblock Diffusion limit for nonhomogeneous and non-micro-reversible
  processes.
\newblock {\em Indiana Univ. Math. J.}, 49(3):1175--1198, 2000.

\bibitem{DuLiu2019}
K.~Du and J.~Liu.
\newblock On the {C}auchy problem for stochastic parabolic equations in
  {H}\"{o}lder spaces.
\newblock {\em Trans. Amer. Math. Soc.}, 371(4):2643--2664, 2019.

\bibitem{DuMeng2010}
K.~Du and Q.~Meng.
\newblock A revisit to {$W_2^n$}-theory of super-parabolic backward stochastic
  partial differential equations in {$\mathbb{R}^d$}.
\newblock {\em Stochastic Processes and their Applications},
  120(10):1996--2015, 2010.

\bibitem{EthierKurtz86}
S.~N. Ethier and T.~G. Kurtz.
\newblock {\em Markov processes}.
\newblock Wiley Series in Probability and Mathematical Statistics: Probability
  and Mathematical Statistics. John Wiley \& Sons Inc., New York, 1986.
\newblock Characterization and convergence.

\bibitem{FouqueGarnierPapanicolaouSolna07}
J.-P. Fouque, J.~Garnier, G.~Papanicolaou, and K.~S{\o}lna.
\newblock {\em Wave propagation and time reversal in randomly layered media},
  volume~56 of {\em Stochastic Modelling and Applied Probability}.
\newblock Springer, New York, 2007.

\bibitem{GarnierSolna16}
J.~Garnier and K.~S{\o}lna.
\newblock Apparent attenuation of shear waves propagating through a randomly
  stratified anisotropic medium.
\newblock {\em Stoch. Dyn.}, 16(4):1650009, 24, 2016.

\bibitem{JacodShiryaev03}
J.~Jacod and A.~N. Shiryaev.
\newblock {\em Limit theorems for stochastic processes}, volume 288 of {\em
  Grundlehren der Mathematischen Wissenschaften [Fundamental Principles of
  Mathematical Sciences]}.
\newblock Springer-Verlag, Berlin, second edition, 2003.

\bibitem{Jakubowski86}
A.~Jakubowski.
\newblock On the {S}korokhod topology.
\newblock {\em Ann. Inst. H. Poincar\'e Probab. Statist.}, 22(3):263--285,
  1986.

\bibitem{Lax2002}
P.~D. Lax.
\newblock {\em Functional analysis}.
\newblock Pure and Applied Mathematics (New York). Wiley-Interscience [John
  Wiley \& Sons], New York, 2002.

\bibitem{Marty06}
R.~Marty.
\newblock On a splitting scheme for the nonlinear {S}chr\"odinger equation in a
  random medium.
\newblock {\em Commun. Math. Sci.}, 4(4):679--705, 2006.

\bibitem{OthmerDunbarAlt1988}
H.~G. Othmer, S.~R. Dunbar, and W.~Alt.
\newblock Models of dispersal in biological systems.
\newblock {\em J. Math. Biol.}, 26(3):263--298, 1988.

\bibitem{OthmerHillen01}
H.~G. Othmer and T.~Hillen.
\newblock The diffusion limit of transport equations. {II}. {C}hemotaxis
  equations.
\newblock {\em SIAM J. Appl. Math.}, 62(4):1222--1250 (electronic), 2002.

\bibitem{PapanicolaouStroockVaradhan77}
G.~C. Papanicolaou, D.~Stroock, and S.~R.~S. Varadhan.
\newblock Martingale approach to some limit theorems.
\newblock In {\em Papers from the {D}uke {T}urbulence {C}onference ({D}uke
  {U}niv., {D}urham, {N}.{C}., 1976), {P}aper {N}o. 6}, pages ii+120 pp. Duke
  Univ. Math. Ser., Vol. III. Duke Univ., Durham, N.C., 1977.

\bibitem{PardouxPiatnitski03}
E.~Pardoux and A.~L. Piatnitski.
\newblock Homogenization of a nonlinear random parabolic partial differential
  equation.
\newblock {\em Stochastic Process. Appl.}, 104(1):1--27, 2003.

\bibitem{Pazy1983}
A.~Pazy.
\newblock {\em Semigroups of {{Linear Operators}} and {{Applications}} to
  {{Partial Differential Equations}}}, volume~44 of {\em Applied Mathematical
  Sciences}.
\newblock {Springer New York}, 1983.

\bibitem{PrevotRockner2007}
C.~Pr\'{e}v\^{o}t and M.~R\"{o}ckner.
\newblock {\em A concise course on stochastic partial differential equations},
  volume 1905 of {\em Lecture Notes in Mathematics}.
\newblock Springer, Berlin, 2007.

\bibitem{Priola99}
E.~Priola.
\newblock On a class of {M}arkov type semigroups in spaces of uniformly
  continuous and bounded functions.
\newblock {\em Studia Math.}, 136(3):271--295, 1999.

\bibitem{SCBPBS11}
J.~Saragosti, V.~Calvez, N.~Bournaveas, B.~Perthame, A.~Buguin, and
  P.~Silberzan.
\newblock Directional persistence of chemotactic bacteria in a traveling
  concentration wave.
\newblock {\em PNAS}, 108(39):16235Ð16240, 2011.

\end{thebibliography}

\end{document}